\documentclass[11pt,oneside,reqno]{amsart}

\usepackage[T1]{fontenc}
\usepackage[utf8]{inputenc}
\usepackage{amssymb}
\usepackage{algpseudocode}
\usepackage{algorithm}
\usepackage{verbatim}
\usepackage{graphicx}
\usepackage{placeins}
\usepackage{listings}
\usepackage{xcolor}
\usepackage{subcaption}
\RequirePackage{dsfont} 
 \scrollmode
\allowdisplaybreaks
\usepackage{graphicx}
\usepackage{epstopdf}
\usepackage{hyperref} 
\usepackage{mathtools}

\usepackage{amsmath}
\usepackage{tikz}


\definecolor{codegreen}{rgb}{0,0.6,0}
\definecolor{codegray}{rgb}{0.5,0.5,0.5}
\definecolor{codepurple}{rgb}{0.58,0,0.82}
\definecolor{backcolour}{rgb}{0.95,0.95,0.92}

\lstdefinestyle{list_style}{
  backgroundcolor=\color{backcolour}, commentstyle=\color{codegreen},
  keywordstyle=\color{magenta},
  numberstyle=\tiny\color{codegray},
  stringstyle=\color{codepurple},
  basicstyle=\ttfamily\footnotesize,
  breakatwhitespace=false,         
  breaklines=true,                 
  captionpos=b,                    
  keepspaces=true,                 
  numbers=left,                    
  numbersep=5pt,                  
  showspaces=false,                
  showstringspaces=false,
  showtabs=false,                  
  tabsize=2
}

\newcommand{\xdasharrow}[2][->]{
\tikz[baseline=-\the\dimexpr\fontdimen22\textfont2\relax]{
\node[anchor=south,font=\scriptsize, inner ysep=1.5pt,outer xsep=2.2pt](x){#2};
\draw[shorten <=3.4pt,shorten >=3.4pt,dashed,#1](x.south west)--(x.south east);
}
}

\newcommand{\DEBUG}{}

\ifdefined\DEBUG%

  \def\rem#1{{\marginpar{\raggedright\scriptsize #1}}}
  \newcommand{\pmr}[1]{\rem{\color{blue}{$\bullet$ #1}}}
  \newcommand{\ppr}[1]{\rem{\color{red}{$\bullet$ #1}}}
 \else%

  \newcommand{\ppr}[1]{}
  \newcommand{\pmr}[1]{}
 \fi

\newcommand{\E}{{\mathbb E}}

\newcommand{\bigo}{\mathcal{O}}

\newcommand{\lip}{\operatorname{Lip}}

\newcommand{\cost}{\operatorname{cost}}

\def\rho{\varrho_1}

\def\rd{\,{\mathrm d}}

\theoremstyle{plain}
\newtheorem{theorem}{Theorem}
\newtheorem{lemma}{Lemma}
\newtheorem{fact}{Fact}

\theoremstyle{definition}

\begin{document}

\title
[On the randomized Euler scheme for SDEs with integral-form drift]
{On the randomized Euler scheme for SDEs with integral-form drift}

\author[P. Przyby{\l }owicz]{Pawe{\l } Przyby{\l }owicz}
\address{AGH University of Krakow,
	Faculty of Applied Mathematics,
	Al. A.~Mickiewicza 30, 30-059 Krak\'ow, Poland}
\email{pprzybyl@agh.edu.pl}

\author[M. Sobieraj]{Micha{\l } Sobieraj}
\address{AGH University of Krakow,
Faculty of Applied Mathematics,
Al. A.~Mickiewicza 30, 30-059 Krak\'ow, Poland}
\email{sobieraj@agh.edu.pl, corresponding author}

\begin{abstract}
In this paper, we investigate the problem of strong approximation of the solutions of stochastic differential equations (SDEs) when the drift coefficient is given in integral form. We investigate its upper error bounds, in terms of the discretization parameter $n$ and the size $M$ of the random sample drawn at each step of the algorithm, in different subclasses of coefficients of the underlying SDE presenting various rates of convergence. Integral-form drift often appears when analyzing stochastic dynamics of optimization procedures in machine learning (ML) problems. Hence, we additionally discuss connections of the defined randomized Euler approximation scheme with the perturbed version of the stochastic gradient descent (SGD) algorithm. Finally, the results of numerical experiments performed using GPU architecture are also reported, including a comparison with other popular optimizers used in ML. 
\end{abstract}

\maketitle


\section{Introduction}
We investigate strong approximation of solutions of the following stochastic differential equations (SDEs) 
\begin{equation}
\label{main_equation}
	\left\{ \begin{array}{ll}
	\displaystyle{
	\rd X(t) = a(X(t))\rd t + b(X(t)) \rd W(t), \ t\in [0,T]},\\
	X(0)=\eta, 
	\end{array} \right.
\end{equation}
where $T \in [0, +\infty)$ and $W$ is a $m$-dimensional Wiener process defined on a complete probability space $(\Omega, \Sigma, \mathbb{P})$ with respect to the filtration $\{\Sigma_t\}_{t \geq 0}$ that satisfies usual conditions. Furthermore,
$\eta\in\mathbb{R}^d$, $b$ is $\mathbb{R}^{d \times m}$-valued mapping, and $a$ is $\mathbb{R}^{d}$-valued mapping such that
\begin{equation}
\label{drift_form}
    a(x)=\mathbb{E}(H(\xi,x))
\end{equation}
for the mapping $H:\mathcal{T} \times\mathbb{R}^d\to\mathbb{R}^d$ and a random element $\xi$ defined on the probability space $(\Omega, \Sigma, \mathbb{P})$ with values in measurable space $(\mathcal{T}, \mathcal{M}, \mu)$, where $\mu$ is the law of $\xi$.
We assume that the $\sigma$-fields $\sigma(\xi)$ and $\sigma(\bigcup\limits_{t \geq 0}\sigma(W(t)))$ are independent.
We are interested in the approximation of the value of $X(T).$


SDEs are the natural extension of Ordinary Differential Equations (ODEs). The usage of the Wiener process allows one to incorporate random fluctuations into the trajectories of the solutions, thus making them random and more irregular. Analytical properties, including the existence and uniqueness of solutions of ODEs and SDEs, are respectively investigated in \cite{barbu} and \cite{cohen_elliott}. 

Such trajectories are able to model data coming from real-life phenomena that include the aforementioned fluctuations. 
Therefore, SDEs are for example widely used in Finance (see \cite{glasserman, MERTON} for option pricing, asset pricing, and applications in risk management) and more generally in time series analysis; see Section 11 in  \cite{sarkka_solin_2019} for the overview of various estimation techniques for SDE parameters from observed time series data. For more applications,  the reader is especially referred to Chapter 7 in \cite{kloeden} and \cite{HARTONO, Iversen, YANG}. 

Regarding the time series analysis, certain classes of ODEs and SDEs can model the trajectories of a descending gradient in the gradient descent algorithms given an optimization problem; see Section 5.2 in \cite{jentzen2023}. On the other hand, the vast majority of optimization problems that arise from Data Science lead to the equations that contain drift in the aforementioned form \eqref{drift_form} because of the random nature of the data; see \cite{pmlr-v40-Ge15}. 
Furthermore, it turns out that approximating the well-known Stochastic Gradient Descent (SGD) algorithm as SDE allows one to benefit from studying a continuous optimization trajectory while carefully preserving the stochasticity of SGD. 
The broad overview of the most popular gradient descent algorithms is provided in \cite{ruder2017overviewgradientdescentoptimization} including the aforementioned SGD algorithm.
For example, see \cite{Laborde2019NesterovsMW} and \cite{NIPS2014_f0969691} for the direct connection of the Nesterov optimization algorithm with ODEs. For the relevant literature connecting SGD with SDEs, the reader is especially referred to the recent advances provided in \cite{Latz, li2021validitymodelingsgdstochastic, malladi2024sdesscalingrulesadaptive, sirignano2017stochasticgradientdescentcontinuous}.

Nonetheless, in \cite{dauphin2014identifyingattackingsaddlepoint}, authors claim that the classical SGD algorithm is vulnerable to stucking in saddle points or at least being slowed down by them. In recent years new variants of SGD algorithms were studied to mitigate the negative impact of the occurrence of saddle points, thus resulting in the introduction of the so-called perturbed stochastic gradient descent (PSGD) algorithm. For the relevant literature on the PSGD algorithm, the reader is especially referred to \cite{pmlr-v80-daneshmand18a, du2017gradientdescentexponentialtime, pmlr-v40-Ge15, pmlr-v70-jin17a, pmlr-v75-jin18a, NEURIPS2018_217e342f}.

In many cases, the existence and uniqueness of the solutions of SDEs are guaranteed but the closed-form formulas are not known. It leads to the usage of numerical schemes for the
approximation of trajectories; see \cite{sauer}.
For papers investigating higher-order methods, the reader is referred to Chapter 17 in \cite{kloeden2}.

In recent years, the properties of so-called randomized Euler algorithms for the approximation of the solutions of differential equations were extensively investigated, which was motivated by the fact that for irregular drift coefficient i.e., discontinuous one w.r.t time variable, classical Euler algorithm fails to converge. See section 3 in \cite{PRZYB2014} for the example that shows a lack of convergence for SDEs in case the drift coefficient is only Borel measurable and the deterministic Euler algorithm is leveraged instead of the randomized one.
Currently, miscellaneous randomized numerical schemes are constructed and investigated on various classes of differential equations.
For example, in \cite{PRZYBYLOWICZ2024143} authors investigate the approximation of solutions to SDDEs under Carathéodory-type drift coefficients. 
On the other hand, research is also focused on point-wise approximation under the assumption that only inexact information about drift and diffusion coefficients is accessible i.e., coefficients are perturbed by some noise; see \cite{morkisz2017} for the construction of randomized Euler scheme under inexact information and investigation of its error and optimality.
In \cite{BOCHACIK2021101554}, authors show that the randomized two-stage Runge–Kutta scheme is the optimal method among all randomized algorithms on certain classes of ODEs based on standard noisy information.
In \cite{BOCHACIK2022}, authors analyze randomized explicit and implicit Euler schemes for ODEs with the right-hand side functions satisfying the Lipschitz condition globally or only locally. 
Finally, in \cite{bochacik2024convergence} authors add randomization to the implicit two-stage Runge-Kutta scheme to improve its rate of convergence and analyze its stability in terms of three notions.
Similarly for the infinite-dimensional setting, in \cite{pprzyb2022} authors introduce a truncated dimension randomized Euler scheme for SDEs driven by an infinite dimensional Wiener process, with additional jumps generated by a Poisson random measure.
Finally, in \cite{pprzyb2024} authors investigate the error of the randomized Milstein algorithm for solving scalar jump-diffusion stochastic differential equations and provide a complete error analysis under substantially weaker assumptions than those known in the literature additionally proving the optimality of the algorithm.

The structure of the paper is as follows. In Section \ref{sec:pre} we introduce the considered class
of SDEs \eqref{main_equation} with admissible coefficients defining the equation.
In Section \ref{sec:grad}, we describe a connection between various gradient descent algorithms and the corresponding variants of Euler schemes especially describing the connection between randomized Euler scheme for SDEs and PSGD algorithm. We stress that this section does not provide any new theoretical findings in terms of optimization but rather points out interesting relations between these two popular topics.
In Section \ref{sec:err}, we define a randomized Euler algorithm for the strong approximation of the solutions of \eqref{main_equation} and investigate its properties, including its upper $L^{p}$-error bound. 
Our theoretical results are supported by numerical
experiments described in Section \ref{sec:numer}. Therefore, we also provide the key elements of our
current CUDA C algorithm implementation. 
As a byproduct, in numerical experiments, we also show that the method can be used to search for the local minima.
Nevertheless, there are no strict theoretical results on finding the actual minima. 
In Section \ref{sec:conc}, we summarize our main findings and list future research.
Finally, in Section \ref{sec:conc}, we state basic facts and provide proofs for the auxiliary lemmas.

\section{Preliminaries}
\label{sec:pre}
We denote by $\mathbb{N}=\{1,2,\ldots\}$. Let  $W = \{W(t)\}_{t\geq 0}$ be a standard $m$-dimensional Wiener process defined on a complete probability space  $(\Omega,\Sigma, \mathbb{P})$. By $\{\Sigma_t\}_{t\geq 0}$ we denote a filtration, satisfying the usual conditions, such that $W$ is a Wiener process with respect to $\{\Sigma_t\}_{t\geq 0}$. We set $\Sigma_{\infty}=\sigma\Bigl(\bigcup_{t\geq 0}\Sigma_t\Bigr)$. Next, let $\xi$ be the random element which is defined on the same probability space $(\Omega, \Sigma, \mathbb{P})$ with values in measurable space $(\mathcal{T}, \mathcal{M}, \mu)$, where $\mu$ is the law of $\xi$.
We assume that the $\sigma$-fields $\sigma(\xi)$ and $\sigma(\bigcup\limits_{t \geq 0}\sigma(W(t)))$ are independent.

By $\|\cdot \|$ we denote Frobenius norm for $\mathbb{R}^{d \times m}$ matrices i.e.,
\begin{equation*}
\|A\|:=\Bigl(\sum\limits_{i=1}^{d}\sum\limits_{j=1}^m |a_{i,j}|^2\Bigr)^{1/2} 
\end{equation*}
where $A=[a_{i,j}]_{i,j=1}^{d,m} \in \mathbb{R}^{d \times m}.$ Furthermore, let $0_{d \times m}$ denote a zero-matrix of size $d \times m,$ and let $I_{m}$ denote identity matrix of size $m \times m.$
Let $p\in [2,+\infty).$ By $L^p(\Omega)$-norm for either a random vector or a random matrix, we mean
\begin{equation*}
 \|Y\|_p := \left( \E\|Y\|^p\right)^{1/p}\qquad\mbox{where}\qquad Y:\Omega \to \mathbb{R}^m \mbox{ or } Y:\Omega \to\mathbb{R}^{d\times m}.
\end{equation*}

For $T \in [0, +\infty)$, we are interested in strong approximation of the solutions $(X(t))_{t \in [0, T]}$ of \eqref{main_equation}. 
The coefficients of the equation are assumed to satisfy the following conditions.

First, we consider only deterministic initial values i.e.,
\begin{itemize}
\item[(A0)] $\eta \in \mathbb{R}^{d}.$ 
\end{itemize}
Next, let $K\in [0, +\infty)$ and let $L: \mathcal{T} \to [0, +\infty)$ be the mapping such that $L \in L^{p}(\mathcal{T}, \mathcal{M}, \mu)$. Similarly we extend previous notation to $\|L\|_{1} := \int\limits_{\mathcal{T}}L(t)\mu(\rd t)$ and $\|L \|_{p} := \|L\|_{{L^P(\mathcal{T}, \mathcal{M}, \mu)}}$. 
The mapping $H:\mathcal{T} \times\mathbb{R}^d\to\mathbb{R}^d$ belongs to $\mathcal{A}(L,K,p)$ iff:
\begin{itemize}
\item[(H1)]$H$ is $\mathcal{M}\otimes\mathcal{B}(\mathbb{R}^{d})/\mathcal{B}(\mathbb{R}^{d})$ measurable, 
\item[(H2)] $\mathbb{E}\|H(\xi, 0)\|^p \leq K,$
\item[(H3)]
$\| H(t,x) - H(t,y) \| \leq L(t) \|x-y\|,$
for $\mu-$almost all $t\in \mathcal{T}$ and all $x,y \in \mathbb{R}^d.$
\end{itemize}
We further define subclasses of $\mathcal{A}(L,K,p)$ based on argument dependence of $H$. 
Let $\mathcal{A}_{1}(K,p)$ denote a subclass of $\mathcal{A}(L, K, p)$ such that $H$ depends only on time variable $t.$ The strict definition of this subclass is
\begin{equation*}
    \mathcal{A}_{1}(L,K,p) := \Big{\{} H \in \mathcal{A}(L, K,p): H(t, x) = H(t, 0) \ \hbox{for} \ \mu\hbox{-almost all} \ t \in \mathcal{T} \ \hbox{and all} \ x \in \mathbb{R}^{d} \Big{\}}.
\end{equation*}
Similarly, let $\mathcal{A}_{2}(L,p)$ denote a subclass of $\mathcal{A}(L, K, p)$ such that $H$ depends only on the spatial variable $x.$ In a similar manner, this subclass can be defined as
\begin{equation*}
    \mathcal{A}_{2}(L,K,p) := \Big{\{} H \in \mathcal{A}(L, K,p): H(t,x) = H(0, x) \ \hbox{for} \ \mu\hbox{-almost all} \ t \in \mathcal{T} \ \hbox{and all} \ x \in \mathbb{R}^{d} \Big{\}}.
\end{equation*}
The admissible drift functions $a:\mathbb{R}^d \to \mathbb{R}^d$ are assumed to be of the form 
\begin{equation*}
a(x)=\mathbb{E}(H(\xi,x))
\end{equation*}
where $H$ belongs to $\mathcal{A}(L,K,p)$ or more strictly one of its subclasses.
A mapping $b:\mathbb{R}^d\to\mathbb{R}^{d\times m}$ belongs to the class of Lipschitz continuous functions $\lip(K)$ iff:
\begin{itemize}
    \item[(B0)] $\|b(x) - b(y)\| \leq K\|x-y\|,$
    for all $x,y \in \mathbb{R}^d.$
\end{itemize}
We define the general class of admissible tuples $(H,b)$ as follows
\begin{equation*}
    \mathcal{F}(L,K,p)=\mathcal{A}(L,K,p)\times\lip(K).
\end{equation*}
We also define the following subclasses of admissible coefficients where diffusion is always identically equal to zero.
Let
\begin{equation*}
\mathcal{F}_{0}^{1}(K,p)= \mathcal{A}_{1}(K,p) \times \{0\},
\end{equation*}
\begin{equation*}
\mathcal{F}_{0}^{2}(L,p)=\mathcal{A}_{2}(L,p) \times \{0\}
\end{equation*}
and
\begin{equation*}
\mathcal{F}_{0}(L,K,p)=\mathcal{A}(L,K,p) \times \{0\}.
\end{equation*}
Finally, we use the following notation of asymptotic equalities. For functions $f,g:\mathbb{N}\times \mathbb{N}\to [0,+\infty)$ we write $ 	f(n,M)=\mathcal{O}(g(n,M))$ iff there exist $C>0, n_0, M_0\in\mathbb{N}$ such that for all $n\geq n_0, \ M \geq M_0$ it holds $f(n,M)\leq Cg(n,M)$. Furthermore, we write $f(n,M)=\Omega(g(n,M))$ iff $ g(n,M)=\mathcal{O}(f(n,M))$,
and $f(n,M)=\Theta(g(n,M))$ iff $f(n,M)=\mathcal{O}(g(n,M)) \ \hbox{and} \ f(n,M)=\Omega(g(n,M))$. Unless otherwise stated all constants appearing in estimates and in the "$\mathcal{O}$", "$\Omega$", "$\Theta$" notation will only depend on $\eta, T$ and the parameters of the class $\mathcal{F}(L, K, p)$. Moreover, the same letter might be used to denote different constants.

%
\section{Connection with gradient descent algorithms}
\label{sec:grad}
In this section, we describe the connection between various variants of gradient descent algorithms and Euler schemes for the approximation of the solutions of differential equations.
In particular, the stochastic gradient descent algorithms are related to Robbins–Monro method, so readers are especially referred to \cite{Robbins1951ASA}. On the other hand, 
we would like to 
\underline{stress} that the aim of this chapter is not to thoroughly review or prove any theoretical results in the context of gradient descent algorithms, but rather point out that there is a connection with strong approximation of SDEs. 
For instance, this connection may lead to interesting example equations that originate from Machine Learning.

First, let us consider the following optimization problem
\begin{equation}
\label{eq_optimization}
\min_{x \in \mathbb{R}^{d}}f(x)
\end{equation}
where $f \in C^{1}(\mathbb{R}^d; \mathbb{R}).$ 
Together with \eqref{eq_optimization}, we associate
the following system of ODEs
\begin{equation}
\label{ode_gradient}
\begin{cases}
& x'(t) = -\nabla f(x(t)), \ t \in [0, +\infty) \\
& x(0) = x_0.
\end{cases}
\end{equation}
Such systems are called gradient systems. See \cite{9992691} for more examples including ODE that models Nesterov’s accelerated gradient optimization algorithm. If
gradient ${\mathbb{R}^d \ni x \mapsto \nabla f(x) \in \mathbb{R}^{d}}$ is continuous and of at most linear growth, then initial-value problem \eqref{ode_gradient} has at least one solution on $[0,T]$ for $T>0$ which is essentially a well-known Peano's theorem. We refer to Theorem (2.12) on page 252 in \cite{andres} where an even more general version of Peano’s theorem was established. If one additionally assumes that $\nabla f$ is locally Lipschitz continuous, then the existence and uniqueness of the solution to \eqref{ode_gradient} can be
derived from Theorem 2.2 on pages 104-105 in \cite{friedman}.

Let $x^{*}\in \mathbb{R}^{d}$ be an isolated critical point of $f$ which is also a local minimum. From Theorem 4.9 in \cite{barbu} (see page 142 for the proof), one obtains that if the initial-value $x_0$ is sufficiently close to $x^{*},$ then $\lim\limits_{t \to +\infty}x(t) = x^{*}$. Approaching
the local minimum $x^{*}$ of $f$ by observing the trajectory $x=x(t)$ while $t \to +\infty$ is sometimes called \textit{continuous learning}.
For $h > 0$ and $t_k = kh, k = 0,1,\ldots,$ the Euler scheme for \ref{ode_gradient} takes the form
\begin{equation*}
\begin{cases}
y_{k+1} = y_k - h \nabla f(y_k), \ k=0,1,\ldots \\
y_0 = x_0.
\end{cases}
\end{equation*}
This is the well-known \textit{gradient descent} (GD) algorithm for searching a (local) minimum of $f$ when starting
from $x_0$. In this context, the step-size $h$ can be interpreted as the learning rate hyperparameter. Using the GD
method to minimize $f$ is also called \textit{discrete learning}.

We discussed the origins of the GD method from the perspective of gradient ODEs. We now show how to obtain its stochastic counterpart. 
Suppose $(\mathcal{T}, \mathcal{M}, \mu)$ is a probability space and $\xi:\Omega \mapsto \mathcal{T}$ is a random element. We usually assume that real-life data samples are independent realizations of $\xi$ which is further called an apriori distribution of the data.
Regarding this random nature of the data, in practice, it usually comes down to the optimization problems which include the expectation formula i.e.,
\begin{equation*}
f(x) = \mathbb{E}(g(\xi, x))
\end{equation*}
is to be optimized whereas $g$ is usually referred to as a loss function; see \cite{pmlr-v40-Ge15}. For example, let $\mathcal{T} = \mathbb{R}^{2}, \mathcal{M} = \mathcal{B}(\mathbb{R}^{2}),$ and let $\mu$ denote any distribution on $(\mathbb{R}^{2}, \mathcal{B}(\mathbb{R}^{2})).$ In supervised learning which is one of the branches of machine learning, one's interest is to find the best mapping $h_x(\cdot),$ depending on some parameters $x \in \mathbb{R}^{d},$ such that $h_{x}(\xi_1) \approx \xi_2$ with the usage of algorithms that successively learn from the data. Especially, $h_{x}(\cdot)$ may denote any artificial neural network with weights $x \in \mathbb{R}^{d}.$ In that case, the approximation error can be measured with mean squared error function i.e., loss function of the form $g(\xi_1, \xi_2, x) = (h_{x}(\xi_1)-\xi_2)^{2}.$ 
Note that usually only a finite number of data samples can be leveraged. Hence, let $(\xi^{i})_{i=1}^{M}$ denote $M$ independent data samples from the common distribution of $\xi.$ The expectation above is therefore approximated with
Monte Carlo sum i.e.,
\begin{equation*}
f(x) \approx \frac{1}{M}\sum_{i=1}^{M}g(\xi^{i}, x)
\end{equation*}
and
\begin{equation*}
\nabla f(x) \approx \frac{1}{M}\sum_{i=1}^{M}\nabla g(\xi^{i}, x), 
\end{equation*}
if $g$ is regular enough. This approach is called \textit{batch gradient descent} as it processes the whole dataset $(\xi^{i})_{i=1}^{M}.$ Furthermore, one can assume that the number of data samples $M$ varies per optimization step which can lead to a smaller memory footprint and potential speedups. Suppose the dataset is split into $k_0\in \mathbb{N}$ sequences $(\xi_{i}^{k+1})_{i=1}^{M_{k}}$ for $k=0,1\ldots,k_0-1,$ called batches. 
In summary, we arrive at the following scheme
\begin{equation*}
\begin{cases}
y_{k+1} = y_k - \frac{h}{M_k}\sum\limits_{i=1}^{M_k}\nabla g(\xi_{i}^{k+1}, y_k), \ k=0,1,\ldots \\
y_0 = x_0.
\end{cases}
\end{equation*}
If $M_k=1$ for all $k=0,1\ldots,k_0-1,$ then the approach is called \textit{stochastic gradient descent} (SGD). On the other hand, if $1<M_k<M,$ it is called \textit{mini batch gradient descent} (MBGD). One may also add some random Gaussian noise in each step to avoid getting stuck in a plateau. Let $(\Delta W_k)_{k=0}^{+\infty}$ denote independent Gaussian random variables with zero-mean and unit-variance, then the scheme can be re-defined as
\begin{equation}
\label{PSGD}
\begin{cases}
y_{k+1} = y_k - \frac{h}{M_k}\sum\limits_{i=1}^{M_k}\nabla g(\xi_{i}^{k+1}, y_k) + \sigma \Delta W_{k}, \ k=0,1,\ldots \\
y_0 = x_0,
\end{cases}
\end{equation}
for some $\sigma > 0,$ and is called \textit{perturbed stochastic gradient descent} (PSGD).
It turns out that the introduced PSGD scheme \eqref{PSGD} is a variant of the randomized Euler scheme for SDEs of the form \eqref{main_equation}, which is introduced and investigated in the next section of this paper; see \cite{lucchi2022theoretical}
where the case of fractional Wiener noise case was considered.

%
\section{Error of the randomized Euler scheme}
\label{sec:err}
In this section, we introduce a randomized Euler scheme for the strong approximation of solutions of \eqref{main_equation} and investigate its $p$-th moments, the algorithm's informational cost, and error upper bounds.

If the explicit formula of $a$ is known, the standard Euler-Maruyama scheme can be leveraged to approximate the value of $X(T)$. The scheme is defined as
\begin{equation}
\label{standard_scheme}
	\begin{cases}
		X_{n}^{E}(0) = \eta, \\
		X_{n}^{E}(t_{k+1}) = X_{n}^{E}(t_{k}) + a(X_{n}^{E}(t_{k}))h + b(X_{n}^{E}(t_k))\Delta W_k, \\
 \quad k=0,1,\ldots, n-1
\end{cases}
\end{equation}
Its $L^{p}(\Omega)$-error, for $p \geq 2,$ is proportional to 
$\bigo(n^{-1/2})$; see Theorem 10.2.2 in \cite{kloeden}.
By the standard Euler-Maruyama algorithm, we mean the evaluation of $X_{n}^{E}(T)$ via the standard Euler-Maruyama scheme.
In this paper, we define the complexity of an algorithm in terms of the information-based complexity framework; see \cite{ibcbook}.
By the informational cost of $X_n^{E}(T)$, we mean the total number of scalar evaluations of $a,b$ and $W$.
Hence, the informational cost of the standard Euler-Maruyama algorithm is $\Theta(n).$

In the next part, we introduce a randomized Euler algorithm that can only access the information about $H$ in case $a$ is unknown. Let $n, M \in \mathbb{N}, h=T/n$ and $t_k = k h$ for $k=0,1. \ldots, n.$ randomized Euler scheme is defined as
\begin{equation}
\label{main_scheme}
	\begin{cases}
		X_{n,M}^{RE}(0) = \eta, \\
		X_{n,M}^{RE}(t_{k+1}) = X_{n,M}^{RE}(t_{k}) + \frac{h}{M} \sum \limits_{j=1}^{M} H(\xi_j^{k+1} , X_{n,M}^{RE}(t_{k})) + b(X_{n,M}^{RE}(t_{k}))\Delta W_k, \\
 \quad k=0,1,\ldots, n-1
	\end{cases}
\end{equation}
where $(\xi_{j}^{k})_{j,k}$ is $\mu$-distributed i.i.d sequence and $\sigma\Big{(}\bigcup\limits_{k\geq 0, j\geq 1} \sigma(\xi_{j}^{k})\Big{)} \perp\!\!\!\perp \sigma\Big{(}\bigcup\limits_{t \geq 0}\mathcal{F}_{t}^{W}\Big{)}$ which means the randomization is independent of driving Wiener process. 
Consequently, \\
$\sigma\Big{(}X_{n,M}^{RE}(t_k)\Big{)} {\perp\!\!\!\perp} \sigma\Big{(} \xi_{1}^{k+1},\ldots,\xi_{M}^{k+1} \Big{)}$ for all $k=1,\ldots,n$; see the first part of the proof of Lemma \ref{lem_a_diff}. 
The continuous version of the randomized Euler scheme is defined as
\begin{equation*}
\tilde{X}_{n,M}^{RE}(t) = \eta + \int\limits_{0}^{t} \sum\limits_{k=0}^{n-1}\frac{1}{M}\sum\limits_{j=1}^{M}H(\xi_j^{k+1}, \tilde{X}_{n,M}^{RE}(t_k))\mathds{1}_{(t_k, t_{k+1}]}(s) \rd s + \int\limits_{0}^{t}\sum\limits_{k=0}^{n-1}b(\tilde{X}_{n,M}^{RE}(t_k))\mathds{1}_{(t_k, t_{k+1}]}(s) \rd W(s)
\end{equation*}
for $t \in [0, T].$ The continuous version of standard Euler scheme $\tilde{X}_{n}^{E}$ is defined analogously. 
From mathematical induction, one obtains that
\begin{equation*}
\tilde{X}_{n,M}^{RE}(t_k) = X_{n,M}^{RE}(t_k)
\end{equation*}
for all $n$ and $M.$ 
Furthermore, from fact \ref{fact_finite_alg}, the randomized Euler scheme takes finite values.
Based on this auxiliary observation, the following Lemma can be formulated.

\begin{lemma}
\label{lemma_finite_val}
There exists constant $C\in [0,+\infty)$ depending only on $\eta, T$ and parameters of the class $\mathcal{F}(L,K,p),$ such that for all $(H, b) \in \mathcal{F}(L,K,P), M,n \in \mathbb{N}$ the following condition holds
\begin{equation*}
\sup\limits_{0 \leq t \leq T} \mathbb{E}\| \tilde{X}_{n,M}^{RE}(t)\|^p \leq C.
\end{equation*}
\end{lemma}

The proof is postponed to the Appendix.

By the randomized Euler algorithm, we mean the evaluation of $X_{n, M}^{RE}(T)$ via the randomized Euler scheme.
Similarly, we define the complexity of the algorithm in terms of the information-based complexity framework. The algorithm evaluates the value of $W$ at $n$ points, $b(\cdot)$ at $n$ points and $H(\cdot, \cdot)$ at $n M$ points. Thus, $\cost(X_{n,M}^{RE}) = 2n + n M = \Theta(nM)$. 
One may notice that the informational cost differs from the informational cost of the standard (randomized) Euler scheme.

Finally, below we provide a theorem that establishes the error upper bounds for the algorithm $X_{n, M}^{RE}.$

\begin{theorem}
\label{thm_upper_bound}
There exist constants $C\in [0, +\infty)$ depending only on $\eta, T$ and parameters of the class $\mathcal{F}(L,K,p)$ or its respective subclasses, such that for all $n,M \in \mathbb{N},$ the following inequalities hold
\begin{itemize}
\item[(E1)]
\begin{equation*}
\| X(T) - X_{n,M}^{RE}(T) \|_{p} \leq C(n^{-1} + M^{-1/2}), \ \mbox{for all} \ (H,b) \in \mathcal{F}_{0}(L,K,p),
\end{equation*}
\item[(E2)]
\begin{equation*}
\| X(T) - X_{n,M}^{RE}(T) \|_{p} \leq C(nM)^{-1/2}, \ \mbox{for all} \ (H,b) \in \mathcal{F}_{0}^{1}(K,p),
\end{equation*}
\item[(E3)]
\begin{equation*}
\| X(T) - X_{n,M}^{RE}(T) \|_{p} \leq Cn^{-1}, \ \mbox{for all} \ (H,b) \in \mathcal{F}_{0}^{2}(L,p),
\end{equation*}
\item[(E4)]
\begin{equation*}
\label{eq_1}
\| X(T) - X_{n,M}^{RE}(T) \|_{p} \leq C(n^{-1/2} + M^{-1/2}), \ \mbox{for all} \ (H,b) \in \mathcal{F}(L,K,p).
\end{equation*}
\end{itemize}
\end{theorem}

To prove Theorem \ref{thm_upper_bound}, we use the following lemma.

\begin{lemma}
\label{lemma_error}
There exists constant $C\in [0, +\infty)$ depending only on $\eta, T$ and parameters of the class $\mathcal{F}(L,K,p),$ such that for all $(H,b) \in \mathcal{F}(L,K,p)$ and $n,M \in \mathbb{N}$ the following inequality holds
\begin{equation*}
\| X_{n}^{E}(T) - X_{n,M}^{RE}(T) \|_{p} \leq C M^{-1/2}.
\end{equation*}
\end{lemma}

\begin{proof}
In this proof, we leverage a continuous versions of Euler schemes. 
First, note that
\begin{equation*}
\mathbb{E}\|\tilde{X}_{n}^{E}(t) - \tilde{X}_{n,M}^{RE}(t) \|^p \leq 2^{p-1}(\mathbb{E}\|A_n^M(t)\|^p + \mathbb{E}\| B_n^M(t)\|^p)
\end{equation*}
where
\begin{equation*}
\mathbb{E}\|A_n^M(t)\|^p  = \mathbb{E}\Big{\|} \int\limits_{0}^{t}\sum\limits_{k=0}^{n-1}\Big{(}\frac{1}{M}\sum\limits_{j=1}^{M}H(\xi_j^{k+1}, \tilde{X}_{n,M}^{RE}(t_k)) - a(\tilde{X}_n^E(t_k))\Big{)}\mathds{1}_{(t_k, t_{k+1}]}(s)\rd s \Big{\|}^p,
\end{equation*}
and
\begin{equation*}
 \mathbb{E}\|B_n^M(t)\|^p = \mathbb{E}\Big{\|} \int\limits_{0}^{t}\sum\limits_{k=0}^{n-1}\Big{(}b(\tilde{X}_{n,M}^{RE}(t_k)) - b(\tilde{X}_{n}^{E}(t_k))\Big{)}\mathds{1}_{(t_k, t_{k+1}]}(s)\rd W(s) \Big{\|}^p.
\end{equation*}
From H\"older's inequality, the fact that intervals $(t_k, t_{k+1}]$ are disjoint for all $k=0,\ldots,n-1,$ equidistant mesh, and finally Fubini's theorem, we obtain that
\begin{equation*}
\begin{split}
\mathbb{E}\|A_n^M(t)\|^p & \leq T^{p-1} \sum\limits_{k=0}^{n-1} \int\limits_{0}^{t}\mathbb{E}\Big{\|} \frac{1}{M}\sum\limits_{j=1}^{M}H(\xi_j^{k+1}, \tilde{X}_{n,M}^{RE}(t_k)) - a(\tilde{X}_n^E(t_k))\Big{\|}^p\mathds{1}_{(t_k, t_{k+1}]}(s)\rd s \\
& \leq 2^{p-1}T^{p-1}\Big{(}A_{1,n}^{M}(t) + A_{2,n}^{M}(t)\Big{)}
\end{split}
\end{equation*}
where
\begin{equation*}
A_{1,n}^{M}(t) = \sum\limits_{k=0}^{n-1} \int\limits_{0}^{t}\mathbb{E}\Big{\|} a(\tilde{X}_n^{RE}(t_k)) - a(\tilde{X}_n^E(t_k))\Big{\|}^p\mathds{1}_{(t_k, t_{k+1}]}(s)\rd s 
\end{equation*}
and
\begin{equation*}
A_{2,n}^{M}(t) = \sum\limits_{k=0}^{n-1} \int\limits_{0}^{t}\mathbb{E}\Big{\|} \frac{1}{M}\sum\limits_{j=1}^{M}H(\xi_j^{k+1}, \tilde{X}_{n,M}^{RE}(t_k)) - a(\tilde{X}_n^{RE}(t_k))\Big{\|}^p\mathds{1}_{(t_k, t_{k+1}]}(s)\rd s .
\end{equation*}
Note that
\begin{equation*}
\begin{split}
A_{1,n}^{M}(t) 
& \leq \| L\|_{1}^{p} \sum\limits_{k=0}^{n-1} \int\limits_{0}^{t}\mathbb{E}\Big{\|} \tilde{X}_{n,M}^{RE}(t_k) - \tilde{X}_n^E(t_k)\Big{\|}^p\mathds{1}_{(t_k, t_{k+1}]}(s)\rd s \\
& \leq \| L\|_{1}^{p} \int\limits_{0}^{t}\sup\limits_{0\leq u \leq s}\mathbb{E}\Big{\|} \tilde{X}_{n,M}^{RE}(u) - \tilde{X}_n^E(u)\Big{\|}^p\rd s
\end{split}
\end{equation*}
and
\begin{equation*}
\begin{split}
A_{2,n}^{M}(t) 
& \leq \sum\limits_{k=0}^{n-1} \int\limits_{0}^{t}\tilde{C}_{p}^p 2^{p-1}\Big{(}1 + \max\limits_{0 \leq i \leq n} \|X_{n,M}^{RE}(t_i) \|_{p}^{p} \Big{)}M^{-p/2}\mathds{1}_{(t_k, t_{k+1}]}(s)\rd s \\
& \leq \tilde{K}_{p} M^{-p/2}
\end{split}
\end{equation*}
which results from lemma \ref{lemma_finite_val} and \ref{lem_a_diff}. Similarly, we obtain from H\"older's inequality and Burkholder inequality that
\begin{equation*}
\begin{split}
\mathbb{E}\|B_n^M(t)\|^p & 
\leq C_p \mathbb{E}\Big{(}\int\limits_{0}^{t}\sum_{k=0}^{n-1}\|b(\tilde{X}_{n,M}^{RE}(t_k)) - b(\tilde{X}_{n}^{E}(t_k))\|^{2}\mathds{1}_{(t_k, t_{k+1}]}(s)\rd s \Big{)}^{p/2}\\
& \leq C_p T^{p/2 - 1}\int\limits_{0}^{t}\sum_{k=0}^{n-1}\mathbb{E}\Big{\|}b(\tilde{X}_{n,M}^{RE}(t_k)) - b(\tilde{X}_{n}^{E}(t_k))\Big{\|}^{p}\mathds{1}_{(t_k, t_{k+1}]}(s)\rd s \\
& \leq
C_p T^{p/2-1} K^{p} \int\limits_{0}^{t}\sup\limits_{0\leq u \leq s}\mathbb{E}\Big{\|} \tilde{X}_{n,M}^{RE}(u) - \tilde{X}_{n}^{E}(u)\Big{\|}^{p}\rd s. \\
\end{split}
\end{equation*}
Henceforth,
\begin{equation*}
\begin{split}
\mathbb{E}\|\tilde{X}_{n}^{E}(t) - \tilde{X}_{n,M}^{RE}(t)\|^p \leq 2^{p-1}\Big{(} & 2^{p-1}T^{p-1}\tilde{K}_{p}M^{-p/2} + \\
& \max\big{\{}T^{p-1}2^{p-1}\|L\|_{1}^p, C_p T^{p/2-1}K^{p}\big{\}}\int\limits_{0}^{t}\sup\limits_{0\leq u \leq s}\mathbb{E}\|\tilde{X}_{n,M}^{RE}(u) - \tilde{X}_{n}^{E}(u)\|^{p} \rd s \Big{)}.
\end{split}
\end{equation*}
Since function function $[0, T]\ni t \mapsto \sup\limits_{0\leq u\leq t} \mathbb{E}\| \tilde{X}_{n,M}^{RE}(u) - \tilde{X}_{n}^{E}(u)\|^p$
is Borel-measurable (as a non-decreasing function) and bounded (by fact \ref{fact_finite_alg}),
applying the Gronwall's
lemma yields
\begin{equation*}
\mathbb{E}\|\tilde{X}_{n}^{E}(t) - \tilde{X}_{n,M}^{RE}(t)\|^p \leq C M^{-p/2}
\end{equation*}
for some $C \in [0, +\infty)$ which completes the proof.
\end{proof}
{\noindent\bf Proof of Theorem \ref{thm_upper_bound}. } 
We provide a proof only for the general case (E4). Inequality (E1) can be proved in a similar manner. On the other hand, (E2) follows immediately from the approximation error of expectation with the Monte Carlo sum. Similarly (E3) follows from the approximation error for the standard Euler algorithm for ODEs. 

From Minkowski inequality, Theorem 10.2.2 in \cite{kloeden} and lemma \eqref{lemma_error}, one obtains that
\begin{equation*}
\begin{split}
\|X(T) - X_{n,M}^{RE}(T)\|_p & \leq \|X(T) - X_{n}^{E}(T)\|_p + \|X_{n}^{E}(T) - X_{n,M}^{RE}(T)\|_p. \\
& \leq C_{1}n^{-1/2} + C_{2}M^{-1/2} \\
& \leq \max\{C_1, C_2\}(n^{-1/2} + M^{-1/2}),
\end{split}
\end{equation*}
which completes the proof for inequality (E4).

\section{Numerical experiments}
\label{sec:numer}
This section compares the obtained theoretical results with the outputs of performed simulations. 
In subsection \ref{numer:err_vs_cost}, we focus on validating the main results from theorem \eqref{thm_upper_bound}. We estimate $L^{2}(\Omega)$ error of the randomized Euler algorithm and compare it with its informational cost.
In subsection \ref{numer:optim}, we inspect randomized Euler algorithm trajectories in the optimization problem of finding local minima. As a byproduct, we show that the method can be used to search for the local minima. We decided to search for the local minima of functions of two variables since such functions can be plotted conveniently.
Next, in subsection \ref{sec:comparison}, we provide a detailed comparison with other optimization algorithms (optimizers), including Nesterov Accelerated Gradient (NAG), Adaptive Gradient (AdaGrad) and Adaptive Moment Estimation (ADAM).
Finally, in subsection \ref{numer:impl}, we provide implementation details in CUDA C. We perform a Monte Carlo simulation with independent samples from the Euler algorithm to estimate the $L^{2}(\Omega)$-error. Hence, to speed up the calculations we leverage CUDA C and parallelize the generation of random samples by running the randomized Euler algorithm on \texttt{Nvidia Titan V (VG100)} GPU under CUDA 10.2. The remaining part of the computations is performed on a single CPU core of \texttt{Intel(R) Xeon(R) CPU E5-2650 v4 @ 2.20GHz}. The implementation utilizes $100$ CUDA blocks and $32$ threads per block.

\subsection{Upper error bound vs informational cost}
\label{numer:err_vs_cost}
First, we studied a system of SDEs with two equations driven by a two-dimensional Wiener process including
$$H(t, x_1, x_2) = (0.4 t^{2} \sin(x_2), 0.8 t^{2} \sin(x_1))^{T}, \ \mbox{with} \ \xi \sim \mathcal{N}(0,1)$$
and
$$b(x_1,x_2)=
\begin{pmatrix}
0.16 x_1 & 0.24 x_2\\ 
0.24 x_1 & 0.32 x_2\\ 
\end{pmatrix}.
$$
The initial condition was $\eta = (1, 1)^{T}.$
We took an estimator of the error $\|X(T)-X_{n, M}^{RE}(T)\|_{2}$ as
\begin{equation}
\label{est_err_1}
	\varepsilon_K( X_{n,M}^{RE}(T)):=\Biggl(\frac{1}{K}\sum_{j=1}^K \|X^{RE, (j)}_{n, M}(T)- X^{RE, (j)}_{100n, 100M}(T)\|^2\Biggr)^{1/2}.
\end{equation}
The figure \ref{fig:ex1} shows results obtained via numerical experiments for $\varepsilon_{6400}( X^{RE}_{n, n}(1))$ and various values of $n.$ Note that, according to the theoretical findings from before, the informational cost of the algorithm is of order $\cost( X^{RE}_{n, n}(1)) = \Theta(n^{2})$, and its $L^{2}(\Omega)$-error is of order $\|X(1) - X^{RE}_{n, n}(1)\|_2 = \bigo(n^{-1/2})$. Hence, $\cost( X^{RE}_{n, n}(1)) = \bigo(\| X(1) - X_{n, n}^{RE}(1)\|_{2}^{-4})$. Similarly, in numerical experiments, we obtain a strong linear correlation between estimated error and cost. It may suggest that in this case, the theoretical findings on the relation between error and cost might be improved and asymptotically equal in terms of $\Theta$ rather than $\bigo$. Nonetheless, this cannot be addressed without any theoretical findings on error lower bounds which we leave as an open question for future research. Finally, it is worth noting that the slope coefficient is equal to $-3.82$ which is also seemingly close to $-4.$ 
\begin{figure*}[h!]
    \centering
    \includegraphics[width=1.0\textwidth]{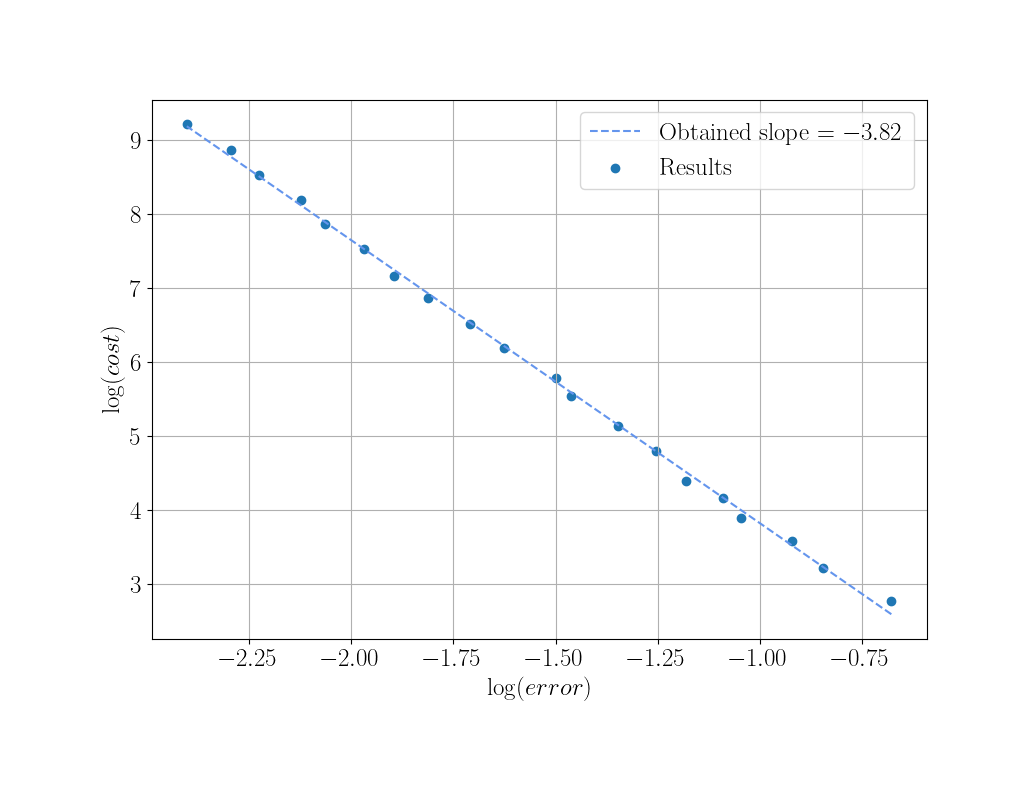}
    \caption{error vs cost}
    \label{fig:ex1}
\end{figure*}

\subsection{Optimization with randomized Euler algorithm}
\label{numer:optim}
In the remaining part of numerical experiments, we inspected trajectories of $X_{100, M}^{RE}(1)$ with respect to $M$ for two optimization problems. 

For the first optimization problem, we took the parabolic loss function 
\begin{equation*}
g(\xi_1, \xi_2, x_1, x_2) = (x_1 - \xi_1)^2 + (x_2 - \xi_2)^2
\end{equation*}
with two independent zero-mean normally distributed random variables $\xi_1, \xi_2.$ Note that 
\begin{equation*}
    f(x_1, x_2) = \mathbb{E}(g(\xi_1, \xi_2, x_1, x_2)) = x_1^2 + x_2^2 + \mathbb{E}(\xi_1^2) + \mathbb{E}(\xi_1^2),
\end{equation*}
and $H(\xi_1, \xi_2, x_1, x_2) = \nabla g(\xi_1, \xi_2, x_1, x_2) = 2(x_1 - \xi_1, x_2 - \xi_2)^{T}.$ Finally, four different trajectories were generated for all initial values $x_0 \in \{-3,3\}\times \{-3, 3\}, b(x_1, x_2) = 0.6 I_{2}$ and $M\in \{1, 5, 10, 100\}$ respectively. In figure \ref{fig:ex_2}, one can notice that the trajectory path stabilizes and starts to resemble the actual gradient descent as $M$ increases.

For the second optimization problem, we introduced a stochastic modification of Himmelblau's function that is of the form
\begin{equation*}
g(\xi_1, \xi_2, x_1, x_2) = (x_1 ^2 + x_2 + \xi_1)^2 + (x_1 + x_2^2 + \xi_2)^2,
\end{equation*}
where $\xi_1 \sim \mathcal{N}(-11, 16), \xi_2 \sim \mathcal{N}(-7, 16)$ and $\xi_1, \xi_2$ are independent. 
Similarly, note that 
\begin{equation*}
    f(x_1, x_2) = \mathbb{E}(g(\xi_1, \xi_2, x_1, x_2)) = (x_1^2 + x_2 - 11)^2 + (x_1 + x_2^2 - 7)^2 + \mathbb{E}(\xi_1^2) + \mathbb{E}(\xi_1^2) - 170,
\end{equation*}
and $H(\xi_1, \xi_2, x_1, x_2) = \nabla g(\xi_1, \xi_2, x_1, x_2).$
We generated two sets of trajectories for $b(x_1, x_2)=0_{2\times2}$ and $M\in \{1, 100\}$ respectively, where each trajectory starts at one of the initial points $x_0 \in \{-1, 1\} \times \{-1, 1\}.$ For distinction, the trajectories generated with $M=100,$ are plotted in red. It turns out, that their final locations coincide with local minima up to three decimal places. In the figure \ref{fig:ex_2_2}, one can also refer to the locations of the final trajectory points for $M=1,$ which are denoted as $x_{n}^{(i)}$ for $i \in \{1,2,3,4\}.$ Note that in contrary to $M=100,$ for $M=1,$ gradient descent trajectories struggle to find respective local minima. 

\clearpage

\begin{figure*}[h!]
    \centering
    \begin{subfigure}[t]{0.5\textwidth}
        \centering
        \includegraphics[width=1.1\textwidth]{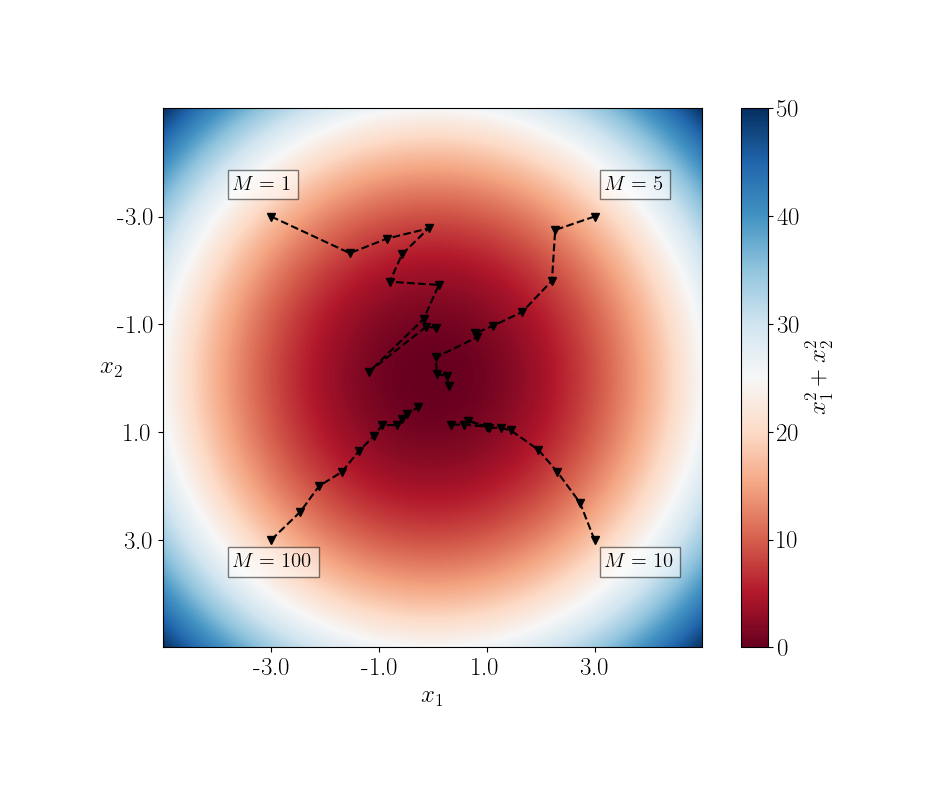}
        \caption{Parabolic loss function}
        \label{fig:ex_2}
    \end{subfigure}%
    ~ 
    \begin{subfigure}[t]{0.5\textwidth}
        \centering
        \includegraphics[width=1.1\textwidth]{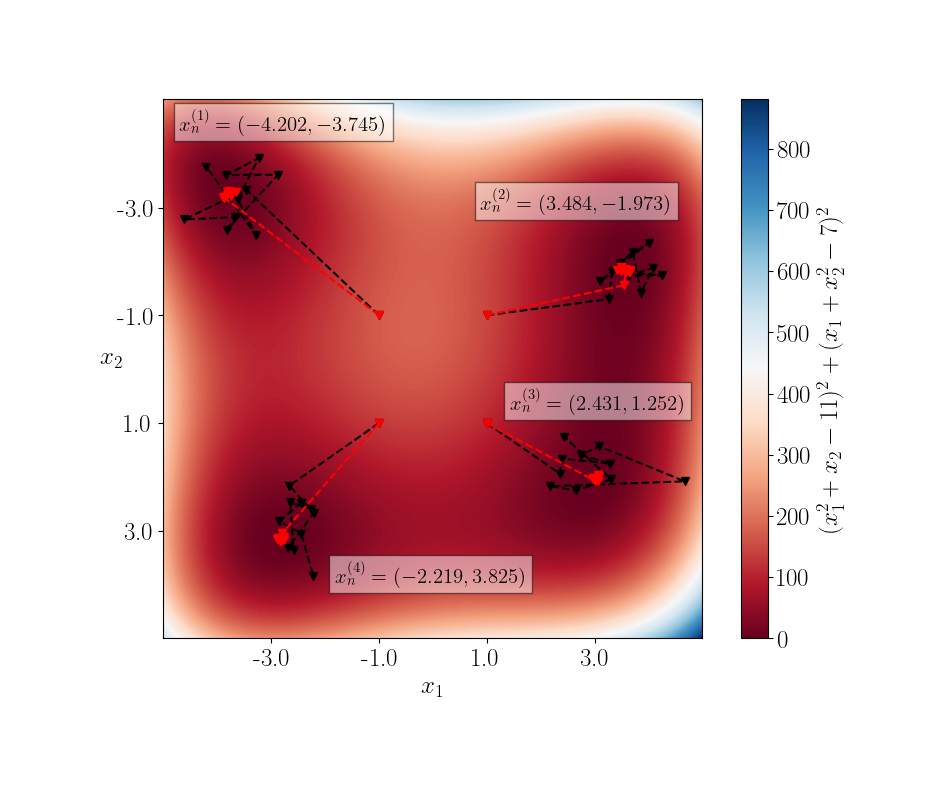}
        \caption{Modified Himmelblau's loss function}
        \label{fig:ex_2_2}
    \end{subfigure}
    \caption{Trajectories searching local minima.}
\end{figure*}

\subsection{Comparison with other optimization methods}
\label{sec:comparison}

In this subsection, we consider $5$ different optimization algorithms/methods, also known as optimizers. In particular, SGD and PSGD optimizers are implemented in line with the Euler scheme definition from section \ref{sec:err}, where PSGD relies on $b(x_1, x_2) = 0.3 I_{2}$.
From the previous subsection, we know that the number of samples can make a real impact on the cost function formula, and henceforth, the optimization method's trajectory. To have reliable comparison results i.e., undistrupted by the impact of the number of samples, we only consider the following deterministic cost functions to be optimized: 
\begin{enumerate}
    \item[(i)]paraboloid function
    \begin{equation*}
        f(x_1, x_2) = x_1^2 + x_2^2,
    \end{equation*}
    \item[(ii)]
    Himmelblau's function
    \begin{equation*}
    f(x_1, x_2) = (x_1^2 + x_2 -11)^2 + (x_1 + x_2^2-7)^2,
    \end{equation*}
    \item[(iii)]
    sigmoid-activation loss function
    \begin{equation*}
    f(x_1, x_2) = [0.5 - (1+\exp(-(x_1+x_2)))^{-1}]^{2}.
    \end{equation*}
\end{enumerate}
\begin{figure*}[h!]
    \centering
    \includegraphics[height=0.5\textwidth]{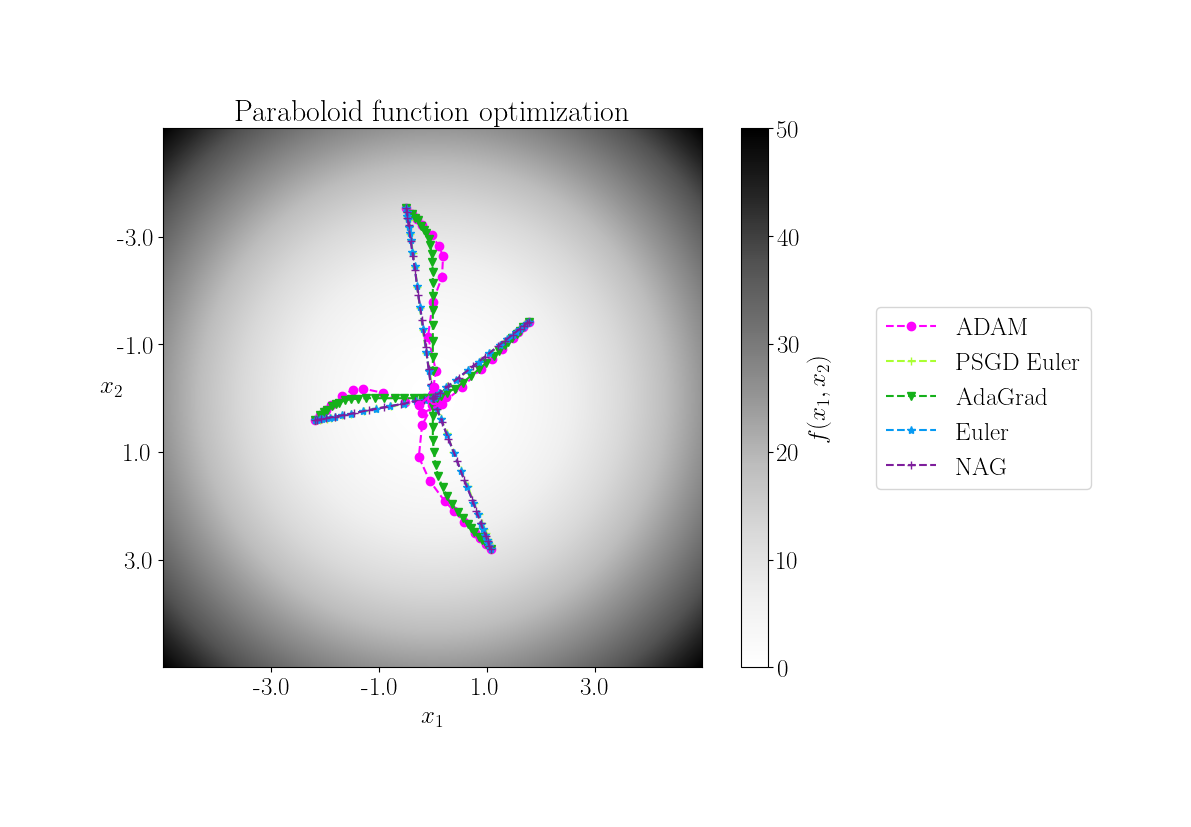}
    \caption{Trajectories of various optimizers searching local minima.}
    \label{fig:fig4}
\end{figure*}
For instance, in figures \ref{fig:fig4}-\ref{fig:fig6}, we can find paths of various optimizers for given optimization problems where all the optimizers start at $4$ common but randomly selected initial points.
Next, we examine optimization errors, which are measured in terms of distance from current point to its closest local minimum at every single iteration. To have a better insight into paths presented in the previous figures, in figures \ref{fig:fig7}-\ref{fig:fig9}, we can find a comparison of errors for one of the previously selected initial points. 
Next, in numerical experiments, we took the following quantitative approach. For each optimization problem, we generated $1000$ random initial values where the distance from each value to its closest local minimum varied between $1.5$ and $3$.
Then mean errors at the log-scale were computed for each optimizer respectively, which can be found in figures \ref{fig:fig10}-\ref{fig:fig12}.
Finally, based on such values, we also computed the quantile spread of log-scale errors for the optimization algorithms that can be found in figures \ref{fig:fig13}-\ref{fig:fig15}.
From the obtained results, we see that the PSGD algorithm works well. Nevertheless, due to its random corrections (in order to avoid plateaus), it is not fully able to find local minima, and henceforth it gets stuck at a certain error level, which can be easily beaten by adaptive methods like ADAM or NAG. In the end, we suspect that a better-designed diffusion function $b$ may lead to finer results, which we leave as an open question.

\begin{figure*}[h!]
    \centering
    \includegraphics[height=0.5\textwidth]{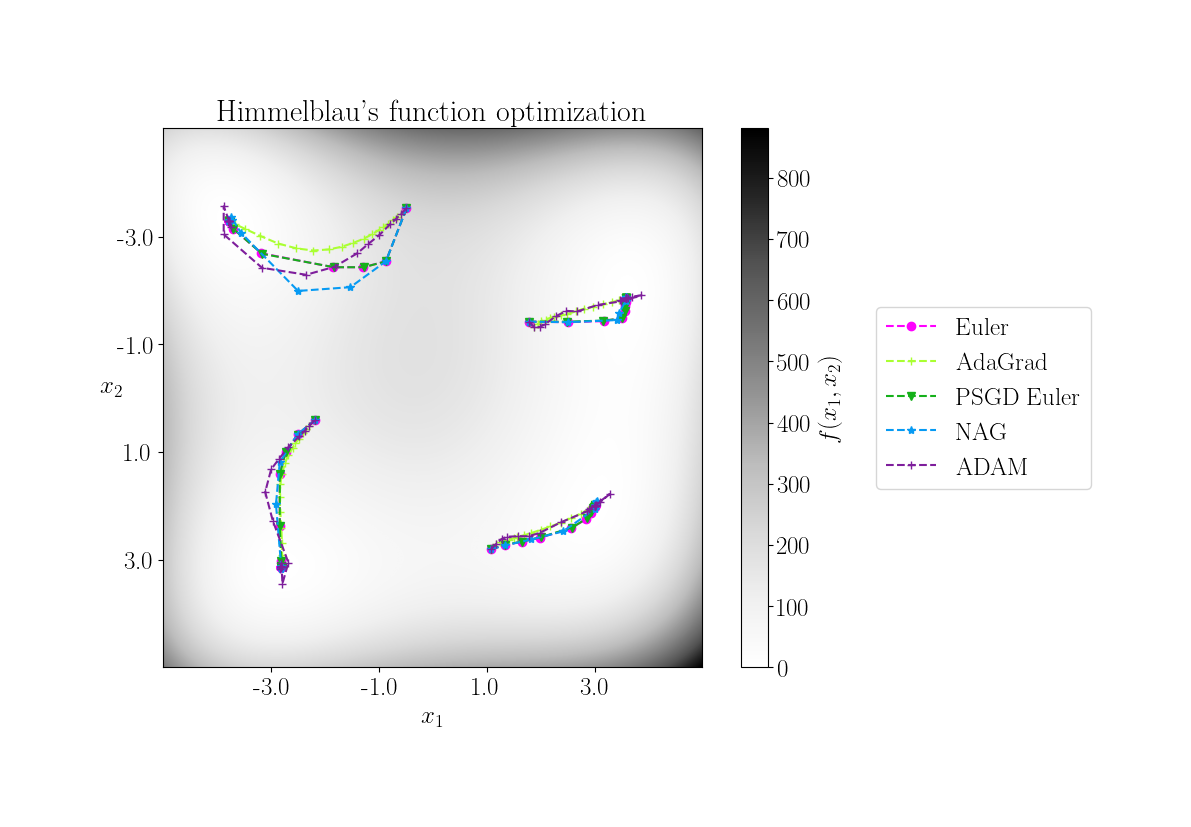}
    \caption{Trajectories of various optimizers searching local minima.}
    \label{fig:fig5}
\end{figure*}

\begin{figure*}[h!]
    \centering
    \includegraphics[height=0.5\textwidth]{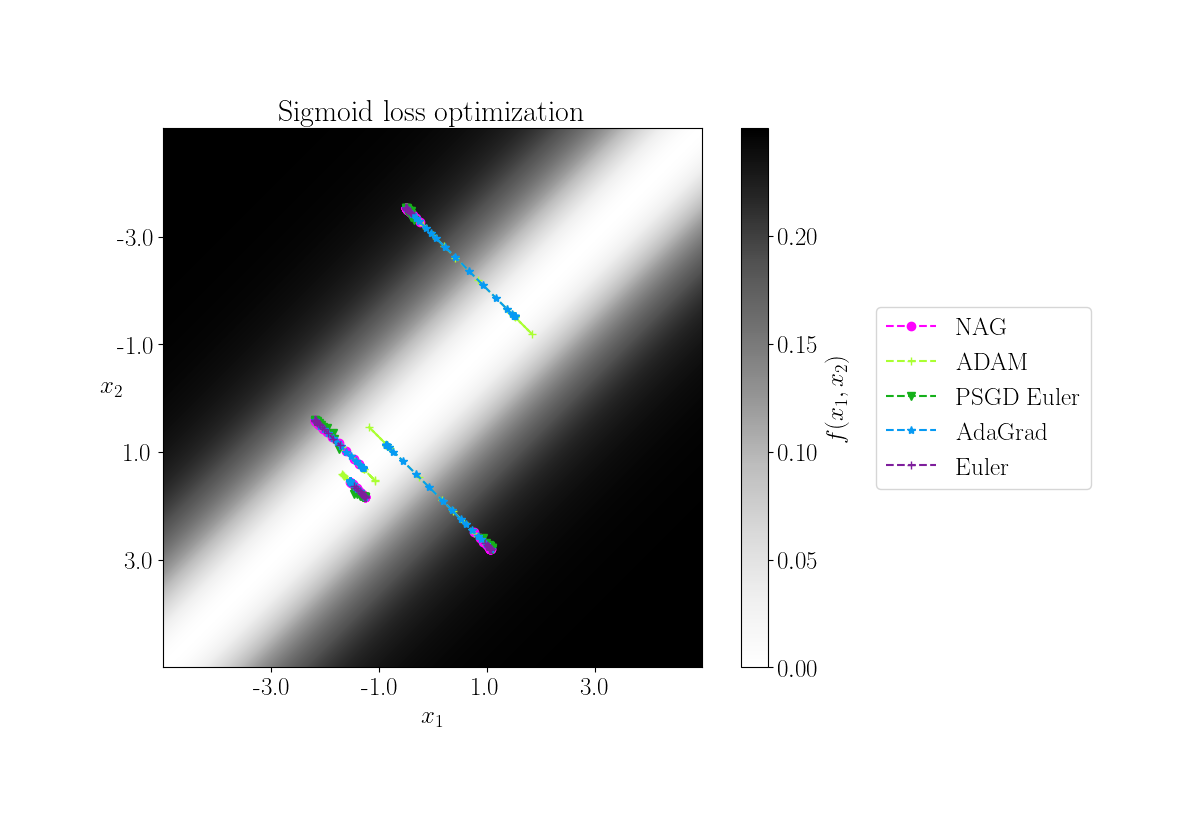}
    \caption{Trajectories of various optimizers searching local minima.}
    \label{fig:fig6}
\end{figure*}

\begin{figure*}[h!]
    \centering
    \includegraphics[height=0.5\textwidth]{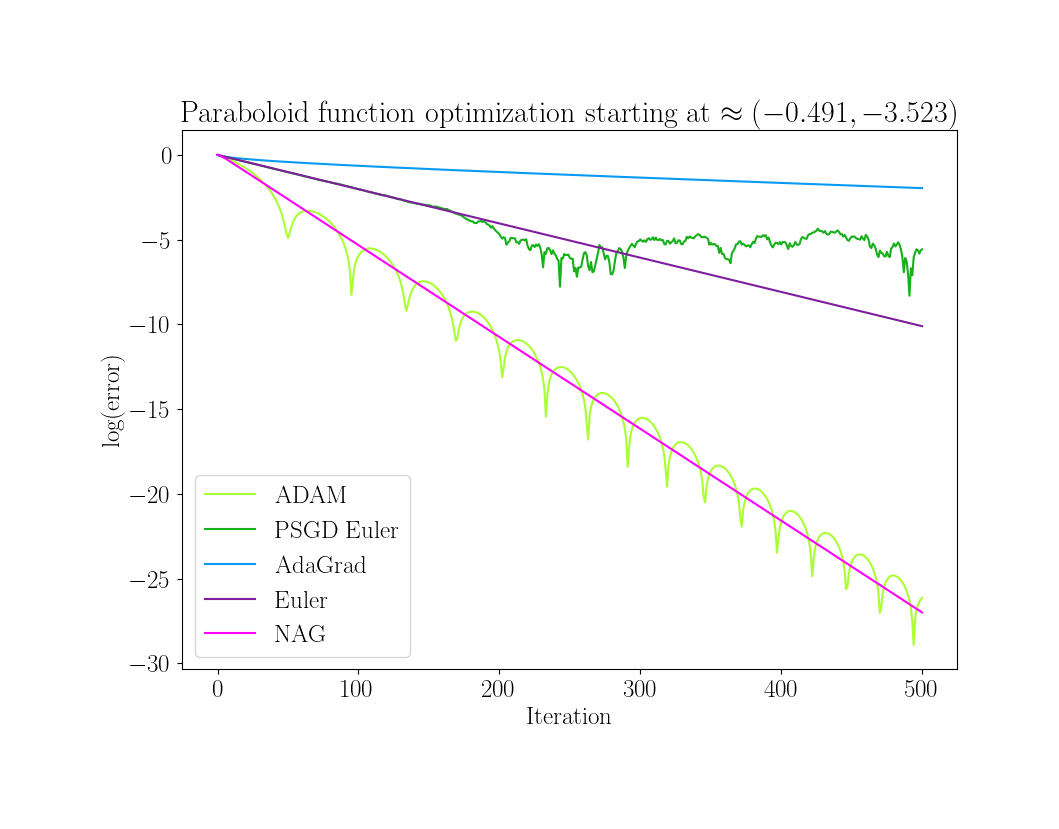}
    \caption{Errors of various optimizers starting at the selected point.}
    \label{fig:fig7}
\end{figure*}

\begin{figure*}[h!]
    \centering
    \includegraphics[height=0.5\textwidth]{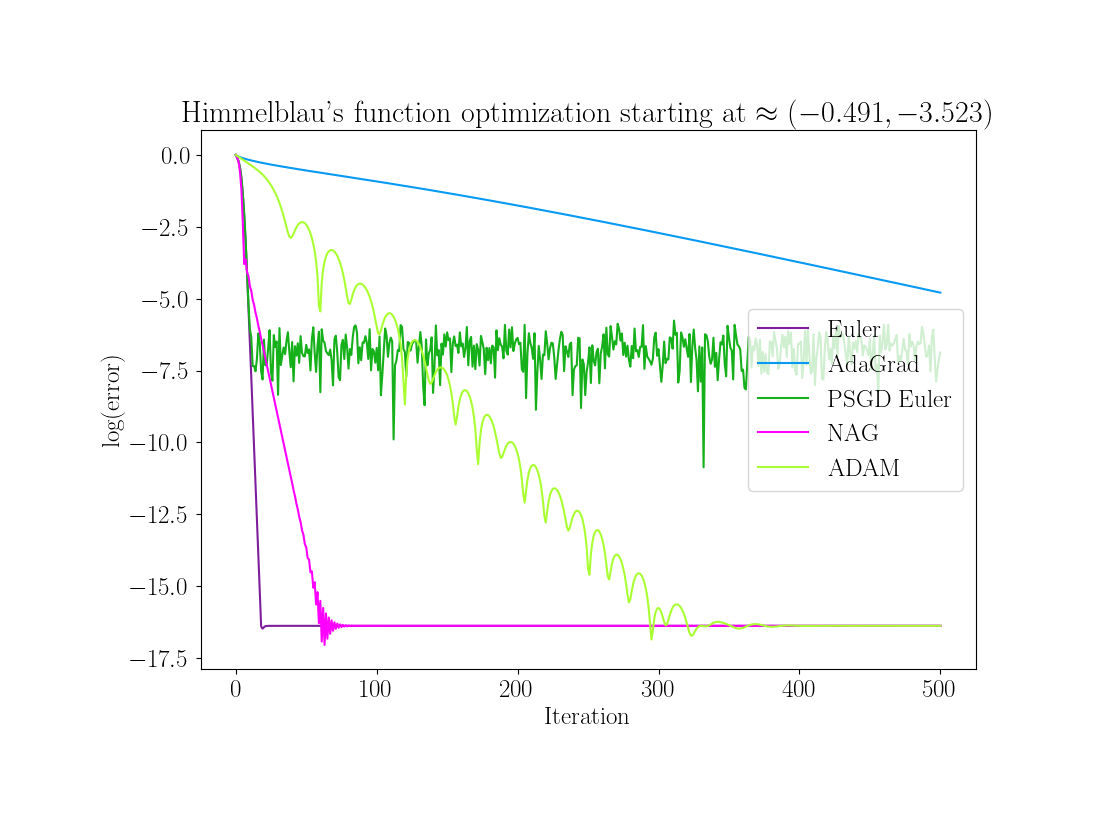}
    \caption{Errors of various optimizers starting at the selected point.}
    \label{fig:fig8}
\end{figure*}

\begin{figure*}[h!]
    \centering
    \includegraphics[height=0.5\textwidth]{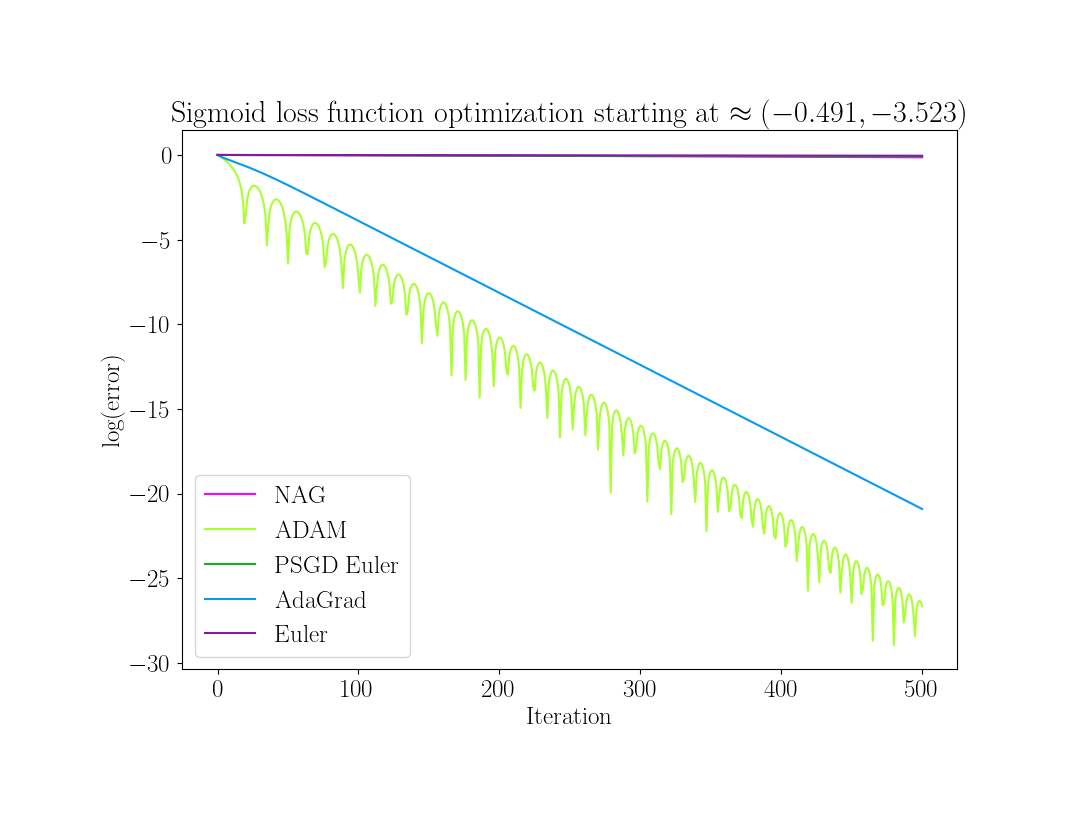}
    \caption{Errors of various optimizers starting at the selected point.}
    \label{fig:fig9}
\end{figure*}

\begin{figure*}[h!]
    \centering
    \includegraphics[height=0.5\textwidth]{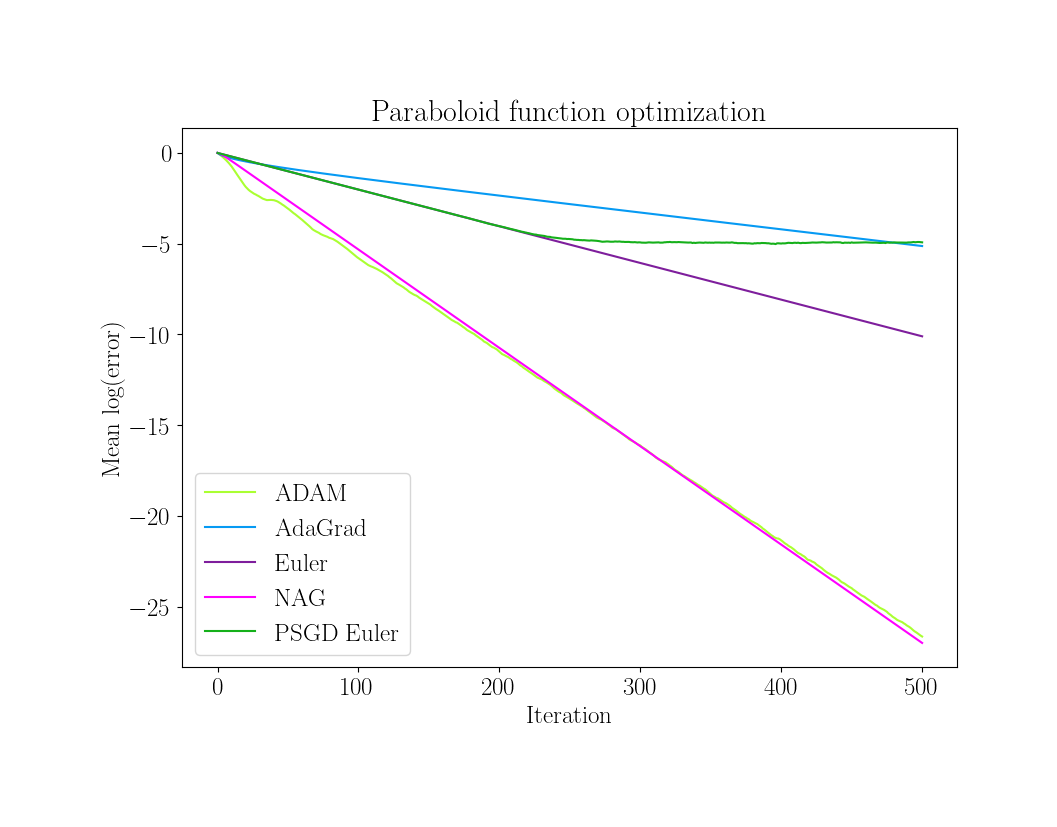}
    \caption{Mean errors of various optimizers starting at randomly selected points.}
    \label{fig:fig10}
\end{figure*}

\begin{figure*}[h!]
    \centering
    \includegraphics[height=0.5\textwidth]{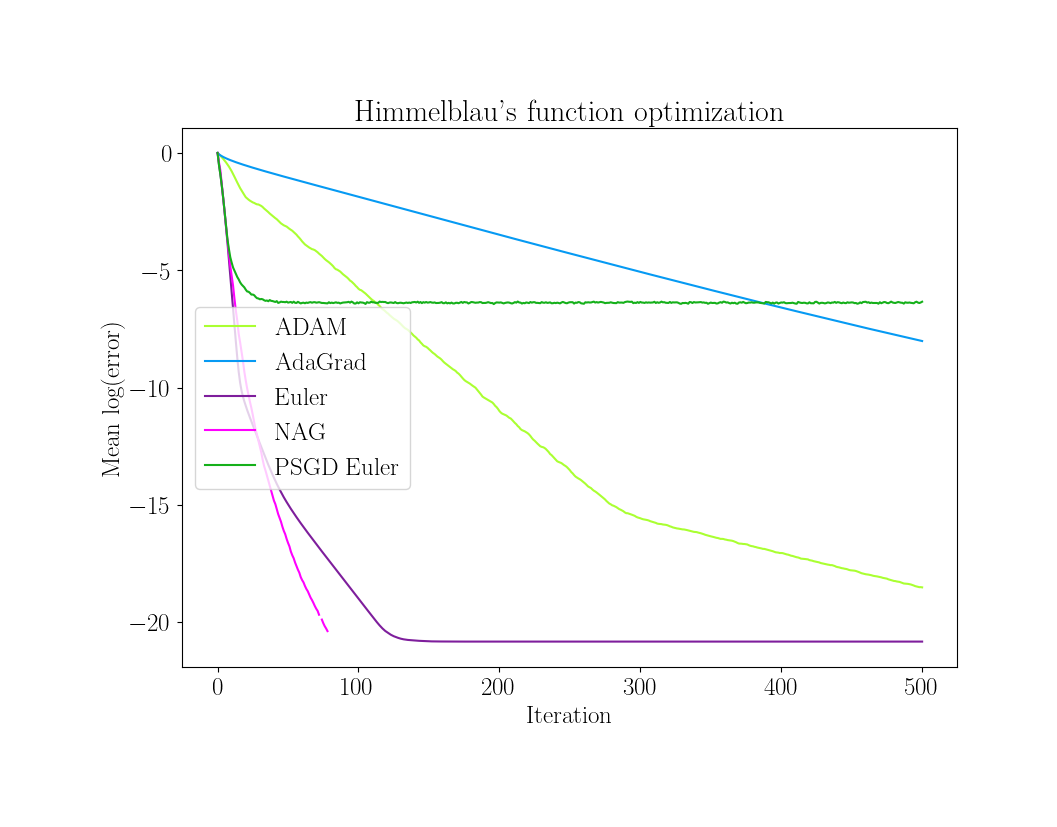}
    \caption{Mean errors of various optimizers starting at randomly selected points.}
    \label{fig:fig11}
\end{figure*}

\begin{figure*}[h!]
    \centering
    \includegraphics[height=0.5\textwidth]{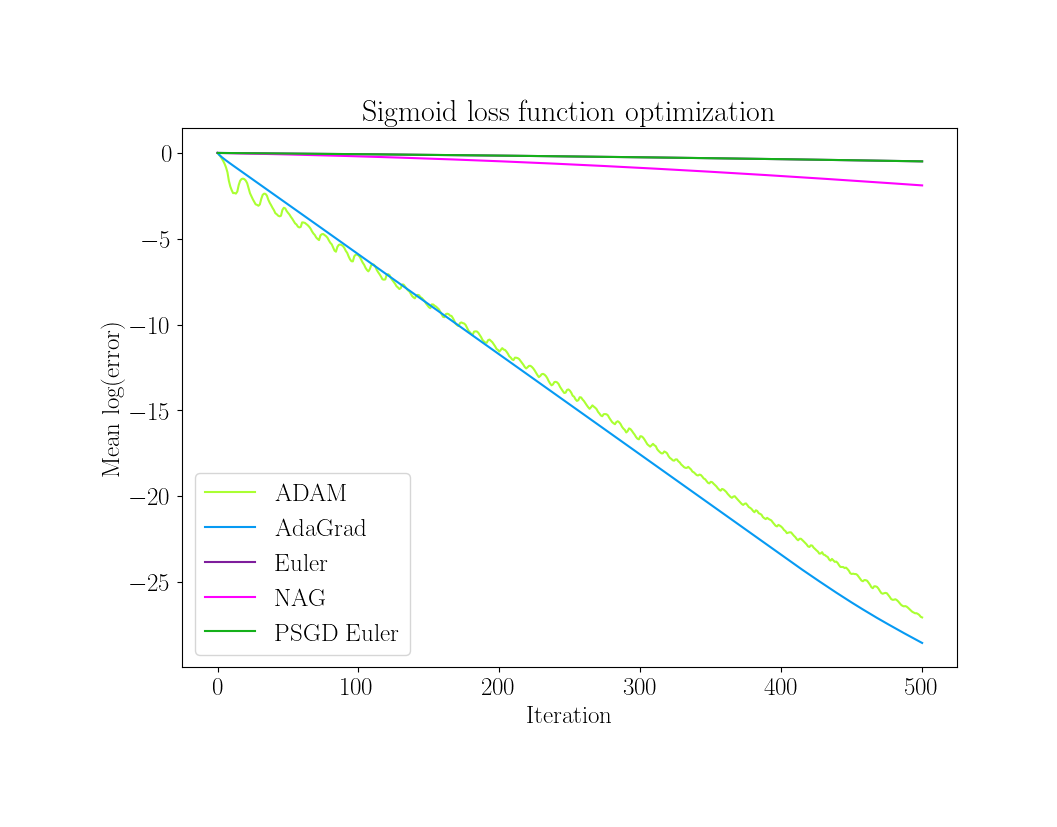}
    \caption{Mean errors of various optimizers starting at randomly selected points.}
    \label{fig:fig12}
\end{figure*}

\begin{figure*}[h!]
    \centering
    \includegraphics[height=0.5\textwidth]{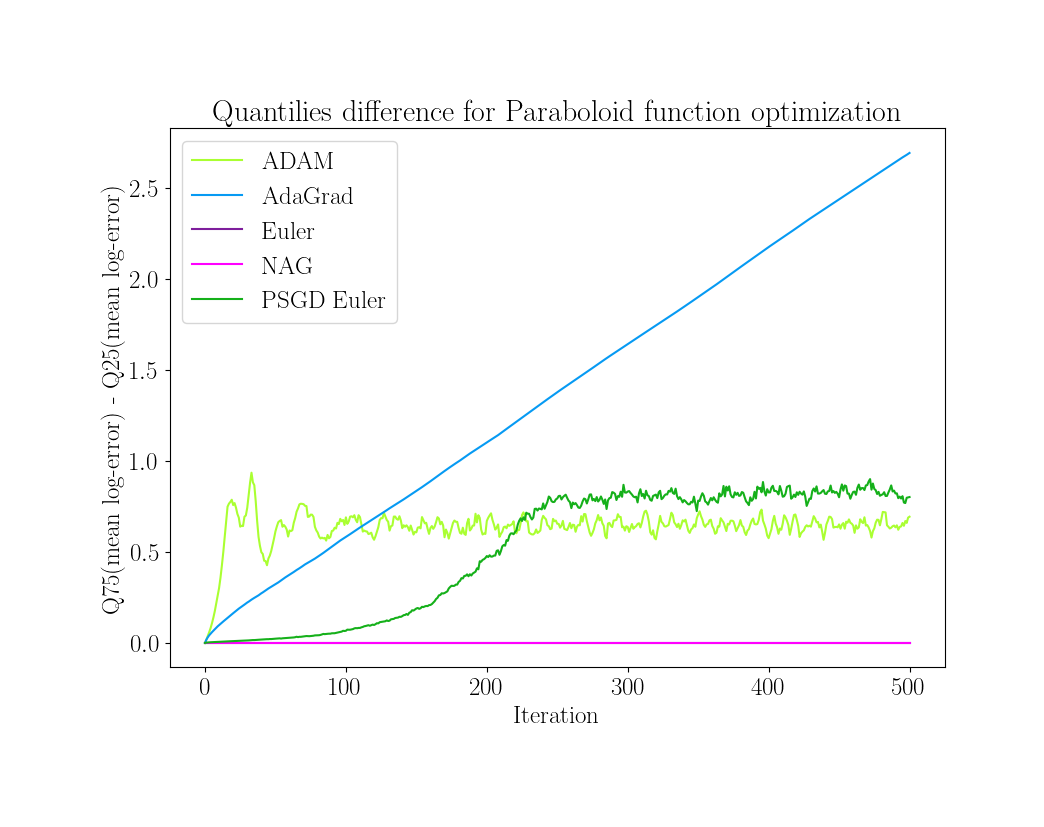}
    \caption{Quantile spread for errors of various optimizers starting at randomly selected points.}
    \label{fig:fig13}
\end{figure*}

\begin{figure*}[h!]
    \centering
    \includegraphics[height=0.5\textwidth]{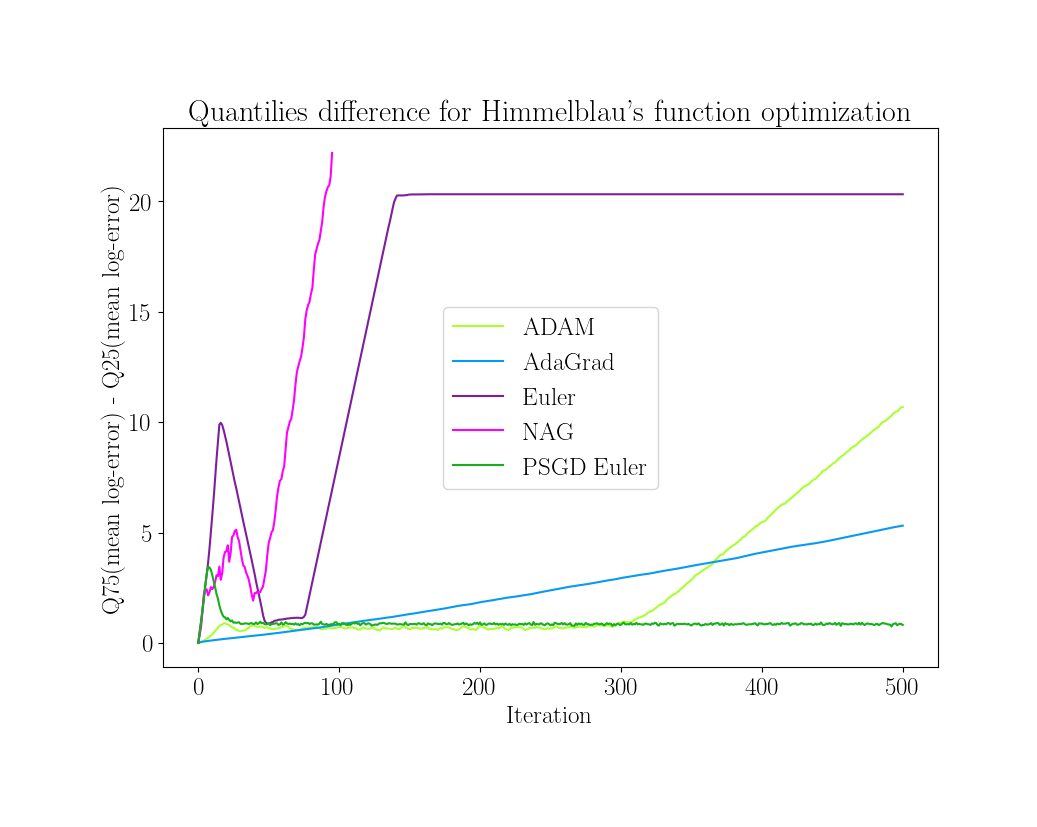}
    \caption{Quantile spread for errors of various optimizers starting at randomly selected points.}
    \label{fig:fig14}
\end{figure*}

\begin{figure*}[h!]
    \centering
    \includegraphics[height=0.5\textwidth]{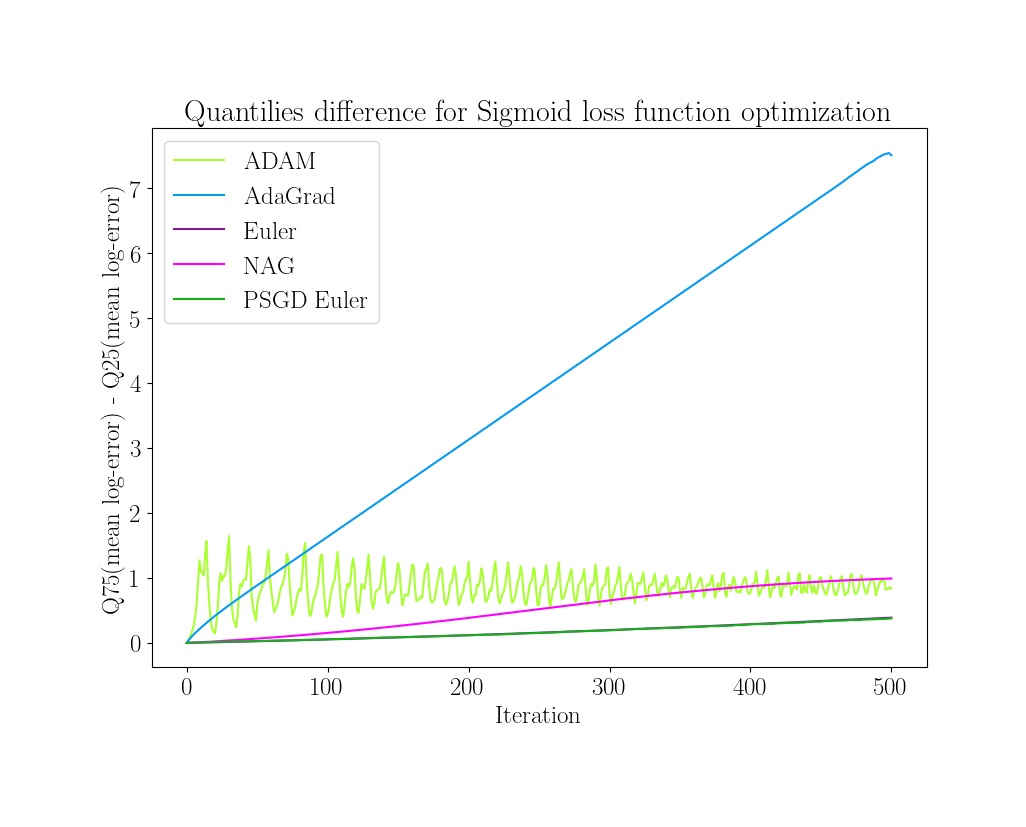}
    \caption{Quantile spread for errors of various optimizers starting at randomly selected points.}
    \label{fig:fig15}
\end{figure*}

\clearpage

\subsection{Implementation details}
\label{numer:impl}

In this subsection, we provide implementation details in CUDA C.


\begin{lstlisting}[language=C++, caption=Numerical scheme implementation, basicstyle=\ttfamily\tiny, label=eulercode]
template <typename type>
__device__ type func_H(int d, type t, type* x);
template <typename type>
__device__ type func_b(int d, int k, type t, type* x);

template <typename type>
void euler_single_step(type t, type* X_prev, int X_dim, type h, type* ksi, int ksi_num, type* dW, int dW_dim, type* X) {
    // (1)
    for (int d = 0; d < X_dim; d++){
        // (2)
        for (int j = 0; j < ksi_num; j++) {
            X[d] += func_H<type>(d, ksi[j], X_prev) * (h/ksi_num);
        }
        // (3)
        for (int k = 0; k < dW_dim; k++) {
            X[d] += func_b<type>(d, k + 1, t, X_prev) * dW[k];
        }
    }
}
\end{lstlisting}
\begin{enumerate}
    \item [(1)] Traverse through system of $d$ equations.
    \item [(2)] Add drift component.
    \item [(3)] Add diffusion component.
\end{enumerate}

\begin{lstlisting}[language=C++, caption=Sampling function, basicstyle=\ttfamily\tiny, label=ylcode]
#define FP float

extern struct curandState_t;
void* malloc(size_t size);
void* memset(void* dest, int ch, size_t count);
void* memcpy(void * destination, const void * source, size_t num);
void free(void *ptr);
__device__ int LCM(int x, int y);
__device__ float curand_normal(curandState_t* state);
__device__ double sqrt(double x);
__device__ float curand_uniform(curandState_t* state);
template <typename type>
__device__ type drift_inverse_cdf(type x);

__device__ void sample_from_coupled_X(int ML, int Ml, int nL, int nl, int dW_dim, FP* x0, FP T, FP* XL, FP* Xl, int X_dim, curandState_t* state) {
    // (1)
    FP* ksi = (FP*)malloc(sizeof(FP) * ML);
    memset(ksi, 0, sizeof(FP) * ML);
    FP* dWL = (FP*)malloc(sizeof(FP) * dW_dim);
    memset(dWL, 0, sizeof(FP) * dW_dim);
    FP* dWl = (FP*)malloc(sizeof(FP) * dW_dim);
    memset(dWl, 0, sizeof(FP) * dW_dim);
    // (2)
    FP tL = 0;
    FP tl = 0;
    memcpy(XL, x0, sizeof(FP)*X_dim);
    memcpy(Xl, x0, sizeof(FP)*X_dim);
    FP* XL_next = (FP*)malloc(sizeof(FP) * X_dim);
    memset(XL_next, 0, sizeof(FP) * X_dim);
    FP* Xl_next = (FP*)malloc(sizeof(FP) * X_dim);
    memset(Xl_next, 0, sizeof(FP) * X_dim);
    // (3)
    int grid_density = LCM(nL, nl);
    FP H = T / grid_density;
    FP t, dW;
    bool compute_XL, compute_Xl;
    for (int i = 0; i < grid_density; i++) {
        t = (i+1)*H;
	      // (4) 
        compute_XL = ((i+1)%(grid_density/nL) == 0);
        compute_Xl = ((i+1)%(grid_density/nl) == 0);
        // (5)
        for (int k = 0; k < dW_dim; k++) {
            dW = (FP)(curand_normal(state) * sqrt(H));
            dWL[k] += dW;
            dWl[k] += dW;
        }
        if (!compute_XL && !compute_Xl) {
            continue; // (6)
        }
        for (int j = 0; j < ML; j++) {
            ksi[j] = drift_inverse_cdf<FP>((FP) curand_uniform(state)); // (7)
        }
        if (compute_XL) {
	          // (8)
            euler_single_step<FP>(tL, XL, X_dim, t-tL, ksi, ML, dWL, dW_dim, XL_next);
            tL = t;
            memcpy(XL, XL_next, sizeof(FP)*X_dim);
            memset(dWL, 0, sizeof(FP)*dW_dim);
        }
        if (compute_Xl) {
	          // (9)
            euler_single_step<FP>(tl, Xl, X_dim, t-tl, ksi, Ml, dWl, dW_dim, Xl_next);
            tl = t;
            memcpy(Xl, Xl_next, sizeof(FP)*X_dim);
            memset(dWl, 0, sizeof(FP) * dW_dim);
        }
   }
   free(ksi);
   free(dWL);
   free(dWl);
   free(XL_next);
   free(Xl_next);
}
\end{lstlisting}
\begin{enumerate}
    \item [(1)] Initialize sparse and dense grid Wiener increments, and randomization.
    \item [(2)] Initialize temporary variables for sparse and dense grid trajectories.
    \item [(3)] Get the least common multiple of grid densities.
    \item [(4)] Check whether the trajectory value can be updated.
    \item [(5)] Update Wiener increments.
    \item [(6)] Skip if neither dense nor sparse trajectory can be updated.
    \item [(7)] Generate drift randomization.
    \item [(8)] Update dense grid trajectory value.
    \item [(9)] Update sparse grid trajectory value.
\end{enumerate}

\section{Conclusions}
\label{sec:conc}
We have investigated the error of the randomized Euler scheme in the case when the drift of underlying SDE has a special integral form. As a result, we proved a theorem that describes the error upper bound of the algorithm in terms of the algorithm parameters i.e., grid density and the number of summands in the approximation of the drift. It turns out that various error upper bounds can be found depending on the argument dependence of drift's integrand. Moreover, we suspect that the informational cost of the algorithm could further be reduced with the usage of the Multilevel Monte Carlo algorithm regarding the approximation of the drift. Nevertheless, we leave this here as an open question for future research.
In the remaining part of the paper, we conducted numerical experiments using GPU architecture. The numerical experiments show that the estimated error conforms with the obtained upper bounds. Hence, we show that the randomized Euler scheme is a suitable method for approximating trajectories of SDEs \eqref{main_equation}. By product, in numerical experiments, we show that the algorithm can be used as an optimization method (due to its connection with the perturbed SGD method). Nevertheless, there are no strict theoretical results on finding the actual local minima. In our future work, we especially plan to go further in exploring the relationship between SDEs and gradient descent algorithms. 
\section{Appendix}
\label{sec:appendix}
\begin{fact}
\label{fact_basic_properties}
The proof of the following fact is straightforward, so we skip the
details.
\begin{enumerate}
\item[i)] Let $H \in \mathcal{A}(L,K,p),$ then
\begin{equation*}
\|H(\xi, x)\|_p \leq \max\{\|H(\xi,0)\|_p, \|L\|_{L^P(\mathcal{T}, \mathcal{M}, \mu)}\}(1 + \| x \|),
\end{equation*}
for all $x \in \mathbb{R}^{d}.$
\item[ii)] Let $H \in \mathcal{A}(L,K,p)$ and let $a(x) = \mathbb{E}(H(\xi, x)),$ then
\begin{equation*}
\|a(x) - a(y)\| \leq \Big{(}\int\limits_{\mathcal{T}}L(t)\mu(\rd t)\Big{)}\| x - y \|,
\end{equation*}
for all $x, y \in \mathbb{R}^d.$
\item[iii)] Let $H \in \mathcal{A}(L,K,p)$ and let $a(x) = \mathbb{E}(H(\xi, x)),$ then
\begin{equation*}
\|a(x)\| \leq \max\{\|H(\xi,0)\|_p, \|L\|_{L^P(\mathcal{T}, \mathcal{M}, \mu)}\}(1 + \| x \|),
\end{equation*}
for all $x \in \mathbb{R}^d.$
\item[iv)]  Let $b \in \lip(K),$ then
\begin{equation*}
\| b(x) \| \leq \max\{\|b(0)\|, K\}(1 + \|x\|),
\end{equation*}
for all $x \in \mathbb{R}^d.$
\end{enumerate}
\end{fact}

\begin{fact}
\label{fact_finite_alg}
For all $n \in \mathbb{N}$ and $ M \in \mathbb{N}$,it holds that
\begin{equation*}
\sup\limits_{0\leq t \leq T}\|\tilde{X}_{n,M}^{RE}(t)\|_p < +\infty
\end{equation*}
\end{fact}
\begin{proof}
Note that it suffices to prove that
\begin{equation*}
\max\limits_{k=0,\ldots,n}\|X_{n,M}^{RE}(t_k)\|_p < +\infty,
\end{equation*}
since $\tilde{X}_{n,M}^{RE}$ is a linear interpolation of $X_{n,M}^{RE}.$
For all $n, M \in \mathbb{N}$ and $k=0,\ldots,n-1,$ random variables $b(X_{n,M}^{RE}(t_k))$ and $\Delta W_k$ are independent. Henceforth,
\begin{equation*}
\begin{split}
\| X_{n,M}^{RE}(t_{k+1})\|_p & \leq  \|X_{n,M}^{RE}(t_{k})\|_p + \frac{h}{M} \sum \limits_{j=1}^{M} \|H(\xi_j^{k+1} , X_{n,M}^{RE}(t_{k}))\|_p + \|b(X_{n,M}^{RE}(t_{k}))\Delta W_k\|_p \\
& = \|X_{n,M}^{RE}(t_{k})\|_p + h \|H(\xi, X_{n,M}^{RE}(t_{k}))\|_p + \|b(X_{n,M}^{RE}(t_{k}))\Delta W_k\|_p \\
& \leq \|X_{n,M}^{RE}(t_{k})\|_p + T \|H(\xi, X_{n,M}^{RE}(t_{k}))\|_p + \|b(X_{n,M}^{RE}(t_{k}))\|_p \|\Delta W_k\|_p.
\end{split}
\end{equation*}
Finally, the thesis follows immediately from the aforementioned observation, mathematical induction, and at most linear growth of $H$ and $b$ from Fact \ref{fact_basic_properties}.
\end{proof}

\begin{lemma}
\label{lem_a_diff}
There exists constant $\tilde{C}_p\in [0,+\infty)$ depending only on $\eta, T$ and parameters of the class $\mathcal{F}(L,K,p),$ such that for all $(H,b) \in \mathcal{F}(L,K,p), M,n \in \mathbb{N}$ and all $k=0,1,...,n-1$ the following inequality holds
\begin{equation*}
\Big{\|} a(X_{n,M}^{RE}(t_k)) - \frac{1}{M}\sum\limits_{j=1}^{M} H(\xi_j^{k+1}, X_{n,M}^{RE}(t_k)) \Big{\|}_{p}
\leq \tilde{C}_p \Big{(}1 + \max\limits_{0 \leq i \leq n} \|X_{n,M}^{RE}(t_i) \|_p \Big{)}M^{-1/2}.
\end{equation*}
\end{lemma}
\begin{proof}
From the fact that $a$ is Borel-measurable, we obtain that $\sigma\Big{(}X_{n,M}^{RE}(t_k)\Big{)} \subset \mathcal{G}_{k}$ for all $k=1,\ldots,n$ where $\mathcal{G}_{k}=\sigma\Big{(} W(t_1),\ldots,W(t_k),\xi_{1}^{1},\ldots,\xi_{M}^{1},\ldots,\xi_1^k,\ldots,\xi_{M}^{k} \Big{)}$ and $\mathcal{G}_{k} \perp\!\!\!\perp \sigma\Big{(} \xi_{1}^{k+1},\ldots,\xi_{M}^{k+1} \Big{)}.$ Moreover, since $\eta$ is deterministic it follows that $\sigma\big{(}X_{n,M}^{RE}(t_0)\big{)} = \{\emptyset, \Omega \} =: \mathcal{G}_{0}.$
Thus, from the properties of conditional expectation, we obtain that
 \begin{equation*}
 \begin{split}
\mathbb{E} \Big{\|} a(X_{n,M}^{RE}(t_k)) - \frac{1}{M}\sum \limits_{j=1}^{M} H(\xi_j^{k+1}, X_{n,M}^{RE}(t_k)) \Big{\|}^{p} & = \mathbb{E} \Big{(} \mathbb{E} \Big{\|} a(X_{n,M}^{RE}(t_k)) - \frac{1}{M}\sum\limits_{j=1}^{M}H(\xi_j^{k+1}, X_{n,M}^{RE}(t_k)) \Big{\|}^p \Big{|} \mathcal{G}_k \Big{)} \\
& = \mathbb{E} \Big{\|} a(x) - \frac{1}{M}\sum\limits_{j=1}^{M} H(\xi_j^{k+1}, x) \Big{\|}^p \Big{|}_{x=X_{n,M}^{RE}(t_k)} \\
& = \mathbb{E} \Big{\|} \int\limits_{\mathcal{T}} H(t,x)\mu(\rd t) - \frac{1}{M}\sum\limits_{j=1}^{M} H(\xi_j^{k+1}, x) \Big{\|}^p \Big{|}_{x=X_{n,M}^{RE}(t_k)}.
\end{split}
\end{equation*}
Note that $\frac{1}{M}\sum\limits_{j=1}^{M} H(\xi_j^{k+1}, x)$ is the approximation of $\int\limits_{\mathcal{T}} H(t,x)\mu(\rd t)$ for all $x \in \mathbb{R}^{d}.$
Let
\begin{equation*}
z_i^{k+1}(x) = \sum\limits_{j=1}^{i}\Big{(}H(\xi_j^{k+1},x)-a(x)\Big{)}
\end{equation*}
for $i\in \{1,\ldots,M\}$ and $z_{0}^{k+1}(x):=0,$ then 
\begin{equation*}
\frac{1}{M}z_M^{k+1}(x) = \frac{1}{M} \sum\limits_{j=1}^{M} H(\xi_j^{k+1},x)- a(x)
\end{equation*}
and
\begin{equation*}
\mathbb{E} \Big{\|} a(x) - \frac{1}{M}\sum\limits_{j=1}^{M} H(\xi_j^{k+1}, x) \Big{\|}^p = \frac{1}{M^p} \mathbb{E} \Big{\|} z_{M}^{k+1}(x) \Big{\|}^{p}.    
\end{equation*}
Let $\mathcal{H}_i^{k+1} := \sigma\Big{(}\xi_1^{k+1},\ldots,\xi_i^{k+1}\Big{)}$ for $k=0,1, \ldots, n-1,$ and let $\mathcal{H}_0^{k+1} := \{\emptyset, \Omega \}.$ Note that
\begin{equation*}
z_i^{k+1}(x) = \Big{(}H(\xi_1^{k+1},x) - a(x)\Big{)} + \ldots + \Big{(} H(\xi_i^{k+1},x) -a(x)\Big{)}
\end{equation*}
and therefore $\sigma\Big{(}z_i^{k+1}(x)\Big{)} \subset \mathcal{H}_i^{k+1}.$
Moreover, since $H(\xi_{i+1}^{k+1},x)$ is $\mathcal{H}_{i}^{k+1}$-independent, we obtain that
\begin{equation*}
\begin{split}
\mathbb{E}\Big{(}z_{i+1}^{k+1}(x) - z_{i}^{k+1}(x) \Big{|}  \mathcal{H}_{i}^{k+1} \Big{)} & 
= \mathbb{E}\Big{(} H(\xi_{i+1}^{k+1}, x) - a(x)  \Big{|}  \mathcal{H}_{i}^{k+1} \Big{)} \\
& = \mathbb{E}\Big{(} H(\xi_{i+1}^{k+1}, x) - a(x) \Big{)} \\
& = \mathbb{E}\Big{(} H(\xi_{i+1}^{k+1}, x)\Big{)} - a(x) = 0
\end{split}
\end{equation*}
which means that $(z_{i}^{k+1}(x), \mathcal{H}_{i}^{k+1})_{i=1,\ldots,M}$ is $\mathbb{R}^{d}$-valued martingale for every $x\in \mathbb{R}^{d}$ and $k=0,1, \ldots, n-1.$
Let $\Big{[}z^{k+1}(x)\Big{]}_{n}$ denote a quadratic variation of $(z_i^{k+1}(x))_{i=0,\ldots,n}.$ Note that
\begin{equation*}
\Big{[}z_i^{k+1}(x)\Big{]}_{n} = \sum\limits_{j=1}^{i}\Big{\|} z_j^{k+1}(x) - z_{j-1}^{k+1}(x) \Big{\|}^{2} = \sum\limits_{j=1}^{i}\Big{\|} H(\xi_j^{k+1},x) - a(x) \Big{\|}^{2}.
\end{equation*}
From Burkholder-Davis-Gundy inequality and Jensen inequality, we obtain that
\begin{equation*}
\begin{split}
\mathbb{E}\| z_{M}^{k+1}(x) \|^{p} \leq \mathbb{E}\Big{[} \max\limits_{0\leq i \leq M} \| z_{i}^{k+1}(x) \| \Big{]}  
& \leq C_p^p \mathbb{E}\Big{[}z^{k+1}(x)\Big{]}_{M}^{p/2} \\
& = C_p^p \mathbb{E} \Big{(} \sum\limits_{j=1}^{M} \| H(\xi_j^{k+1}, x) - a(x) \|^{2} \Big{)}^{p/2} \\
& \leq C_p^p M^{p/2 - 1} \sum\limits_{j=1}^{M}\mathbb{E}\| H(\xi_j^{k+1}, x) - a(x) \|^{p} \\
& \leq C_p^p M^{p/2} \mathbb{E}\| H(\xi, x) - a(x) \|^{p}
\end{split}
\end{equation*}
where $C_{p}^{p}\in [0, +\infty)$ does not depend on $k$ and $x.$ Let $K_p(x) := C_p \| H(\xi, x) - a(x) \|_{L^p}.$
Then,
\begin{equation*}
\mathbb{E} \Big{\|} a(x) - \frac{1}{M}\sum\limits_{j=1}^{M} H(\xi_j^{k+1}, x) \Big{\|}^p = \frac{\mathbb{E}\| z_{M}^{k+1}(x)\|^{p}}{M^{p}} \leq K_{p}^{p}(x)M^{-p/2},
\end{equation*}
and
\begin{equation*}
K_{p}(x) \leq C_p\|H(\xi,x)\|_{L^p} + C_p \| a(x) \|_{L^{p}}.
\end{equation*}
If we assume that $\int\limits_{\mathcal{T}}(L(t))^p\mu(\rd t) < +\infty,$ then
\begin{equation*}
\begin{split}
\Big{|} \| H(\xi, x) - H(\xi, 0) \|_{L^{p}} \Big{|} 
& \leq \| H(\xi, x) - H(\xi, 0) \|_p \\
& = \Big{(} \mathbb{E} \|H(\xi,x) - H(\xi,0) \|^{p} \Big{)}^{1/p} \\
& = \Big{(} \int\limits_{\mathcal{T}} \|H(\xi,x) - H(\xi,0) \|^p \mu(\rd t) \Big{)}^{1/p} \\
& \leq \Big{(} \int\limits_{\mathcal{T}} (L(t))^p \mu(\rd t) \|x\|^p \Big{)}^{1/p} \leq \| L \|_{p} \|x\|^p.
\end{split}
\end{equation*}
Moreover, if $\mathbb{E}\|H(\xi, 0)\|^p < +\infty,$ then there exists constant $K_1\in [0, +\infty)$ (see fact \ref{fact_basic_properties}) depending only on the parameters of the class $\mathcal{F}(L,K,p),$ such that 
\begin{equation*}
\|H(\xi,x)\|_{p} \leq \|H(\xi,0)\|_{p} + \|L\|_p\|x\| \leq K_1(1+\|x\|)
\end{equation*}
and therefore $0 \leq K_p(x) \leq \tilde{C}_p(1+ \| x \|)$ which results in
\begin{equation*}
\Big{\|} a(X_{n,M}^{RE}(t_k)) - \frac{1}{M}\sum\limits_{j=1}^{M} H(\xi_j^{k+1}, X_{n,M}^{RE}(t_k)) \Big{\|}_{p}
\leq \tilde{C}_p \Big{(}1 + \max\limits_{0 \leq i \leq n} \|X_{n,M}^{RE}(t_i) \|_p \Big{)}M^{-1/2}.
\end{equation*}    
\end{proof}

{\noindent\bf Proof of Lemma \ref{lemma_finite_val}. }
In this proof, we proceed similarly as in the proof of lemma \ref{lemma_error} leveraging a continuous version of the Euler algorithm. 
First, note that
\begin{equation*}
\mathbb{E}\|\tilde{X}_{n,M}^{RE}(t) \|^p \leq 3^{p-1}(\|\eta\|^p + \mathbb{E}\|\tilde{A}_n^M(t)\|^p + \mathbb{E}\| \tilde{B}_n^M(t)\|^p)
\end{equation*}
where
\begin{equation*}
\mathbb{E}\|\tilde{A}_n^M(t)\|^p  = \mathbb{E}\Big{\|} \int\limits_{0}^{t}\sum\limits_{k=0}^{n-1}\Big{(}\frac{1}{M}\sum\limits_{j=1}^{M}H(\xi_j^{k+1}, \tilde{X}_{n,M}^{RE}(t_k))\Big{)}\mathds{1}_{(t_k, t_{k+1}]}(s)\rd s \Big{\|}^p,
\end{equation*}
and
\begin{equation*}
 \mathbb{E}\|\tilde{B}_n^M(t)\|^p = \mathbb{E}\Big{\|} \int\limits_{0}^{t}\sum\limits_{k=0}^{n-1}b(\tilde{X}_{n,M}^{RE}(t_k))\mathds{1}_{(t_k, t_{k+1}]}(s)\rd W(s) \Big{\|}^p.
\end{equation*}
From H\"older's inequality, the fact that intervals $(t_k, t_{k+1}]$ are disjoint for all $k=0,\ldots,n-1,$ equidistant mesh, and finally Fubini's theorem, we obtain that
\begin{equation*}
\begin{split}
\mathbb{E}\|\tilde{A}_n^M(t)\|^p & \leq T^{p-1} \sum\limits_{k=0}^{n-1} \int\limits_{0}^{t}\mathbb{E}\Big{\|} \frac{1}{M}\sum\limits_{j=1}^{M}H(\xi_j^{k+1}, \tilde{X}_{n,M}^{RE}(t_k))\Big{\|}^p\mathds{1}_{(t_k, t_{k+1}]}(s)\rd s .
\end{split}
\end{equation*}
Note that, from H\"older's inequality and i) in fact \ref{fact_basic_properties}, we obtain 
\begin{equation*}
\begin{split}
\mathbb{E}\Big{\|} \frac{1}{M}\sum\limits_{j=1}^{M}H(\xi_j^{k+1}, \tilde{X}_{n,M}^{RE}(t_k))\Big{\|}^p & \leq \frac{1}{M}\sum\limits_{j=1}^{M}\mathbb{E}\Big{\|} H(\xi_j^{k+1}, \tilde{X}_{n,M}^{RE}(t_k))\Big{\|}^p \\
& \leq 2^{p-1}C_{1}^{p}\Big{(}1 + \mathbb{E}\|\tilde{X}_{n,M}^{RE}(t_k)\|^{p}\Big{)}.
\end{split}
\end{equation*}
Thus,
\begin{equation*}
\begin{split}
\mathbb{E}\|\tilde{A}_{n}^{M}(t)\|^{p} 
& \leq T^{p-1}2^{p-1}C_{1}^{p} \sum\limits_{k=0}^{n-1} \int\limits_{0}^{t} \Big{(}1 + \mathbb{E}\|\tilde{X}_{n,M}^{RE}(t_k)\|^{p}\Big{)} \mathds{1}_{(t_k, t_{k+1}]}(s)\rd s \\
& \leq K_1 + K_2\int\limits_{0}^{t} \sup\limits_{0\leq u \leq s}\mathbb{E}\|\tilde{X}_{n,M}^{RE}(u)\|^{p}\rd s .
\end{split}
\end{equation*}
In a similar fashion,
we obtain from H\"older's inequality and Burkholder inequality that
\begin{equation*}
\begin{split}
\mathbb{E}\|\tilde{B}_n^M(t)\|^p & 
\leq \tilde{C}_p \mathbb{E}\Big{(}\int\limits_{0}^{t}\sum_{k=0}^{n-1}\|b(\tilde{X}_{n,M}^{RE}(t_k))\|^{2}\mathds{1}_{(t_k, t_{k+1}]}(s)\rd s \Big{)}^{p/2}\\
& \leq \tilde{C}_p T^{p/2 - 1}\int\limits_{0}^{t}\sum_{k=0}^{n-1}\mathbb{E}\|b(\tilde{X}_{n,M}^{RE}(t_k))\|^{p}\mathds{1}_{(t_k, t_{k+1}]}(s)\rd s \\
& \leq \tilde{C}_p T^{p/2 - 1}2^{p-1}C_{2}^{p}\int\limits_{0}^{t}\sum_{k=0}^{n-1}\Big{(}1+\mathbb{E}\|\tilde{X}_{n,M}^{RE}(t_k)\|^{p}\Big{)}\mathds{1}_{(t_k, t_{k+1}]}(s)\rd s \\
& \leq
K_3 + K_4 \int\limits_{0}^{t}\sup\limits_{0\leq u \leq s}\mathbb{E}\Big{\|} \tilde{X}_{n,M}^{RE}(u)\Big{\|}^{p}\rd s.
\end{split}
\end{equation*}
Henceforth,
\begin{equation*}
\begin{split}
\mathbb{E} \|\tilde{X}_{n,M}^{RE}(t)\|^p \leq 3^{p-1}\Big{(}\|\eta \|^{p} + K_1 + K_3 + (K_2 + K_4)\int\limits_{0}^{t} \sup\limits_{0\leq u \leq s}\mathbb{E}\|\tilde{X}_{n,M}^{RE}(u)\|^{p}\rd s  \Big{)}.
\end{split}
\end{equation*}
Since function $[0, T]\ni t \mapsto \sup\limits_{0\leq u\leq t} \mathbb{E}\| \tilde{X}_{n,M}^{RE}(u)\|^p$
is Borel-measurable (as a non-decreasing function) and bounded (by fact \ref{fact_finite_alg}),
applying the Gronwall's
lemma yields
\begin{equation*}
\sup\limits_{0 \leq s \leq t} \mathbb{E}\| \tilde{X}_{n,M}^{RE}(t)\|^p \leq C
\end{equation*}
for some $C \in [0, +\infty)$ which depends only on $\eta,T$ and parameters of the class $\mathcal{F}(L,K,p).$

\noindent\newline\newline {\bf Acknowledgements}
We gratefully acknowledge the support of the Leibniz Center for Informatics,
where several discussions about this research were held during the Dagstuhl Seminar "Algorithms and Complexity for Continuous Problems" (Seminar ID 23351). We want to thank the unknown reviewers for their valuable comments and suggestions on improving the overall quality of this work.

\noindent\newline\newline {\bf Statements and Declarations}

The authors declare no competing interests.
No funding was received to assist with the preparation of this manuscript.
The datasets generated during and/or analysed during the current study are available from the authors on reasonable request.

\bibliographystyle{siam}
\bibliography{mybib}

\end{document}